\documentclass[11pt]{article}

\usepackage[T1]{fontenc}
\usepackage{mathpazo} 
\usepackage{microtype} 
\usepackage{amsmath,amstext,amsopn,amssymb, amsfonts, amsthm, mathtools} 
\usepackage{enumitem} 
\usepackage{geometry} 
\usepackage[dvipsnames]{xcolor}
\usepackage{bbm}
\usepackage{xfrac}
\usepackage{multicol}
\usepackage{graphicx}
\usepackage{subcaption}
\usepackage{mathrsfs}
\usepackage{pifont}
\usepackage{comment}
\usepackage[colorlinks=true, linkcolor=spurple, urlcolor=sblue, citecolor=spurple]{hyperref}

\usepackage{wrapfig}
\usepackage[font=small,labelfont=bf]{caption} 

\setlist[itemize]{leftmargin=2em, itemsep=0.4em, topsep=0.6em, label=\scalebox{0.6}{\ensuremath{\blacksquare}}}

\usepackage[font=footnotesize,labelfont=bf]{caption}

\newcommand{\unit}{1\!\!1}

\newcommand{\mbz}{\mathbb{Z}}
\newcommand{\mbr}{\mathbb{R}}
\newcommand{\mbn}{\mathbb{N}}

\newcommand{\mbg}{\mathbb{G}}

\newcommand{\mcd}{\mathcal{D}}
\newcommand{\mca}{\mathcal{A}}

\newcommand{\mcg}{\mathcal{G}}

\newcommand{\scs}{\mathscr{S}}
\newcommand{\scg}{\mathscr{G}}

\newcommand{\La}{\langle }
\newcommand{\Ra}{\rangle }
\newcommand{\mfs}{\mathfrak{S}}

\newcommand{\bfg}{\mathbf{\widetilde{G}}}

\usepackage{xcolor}
\definecolor{spurple}{RGB}{148, 0, 211}
\definecolor{sblue}{RGB}{0, 102, 255}
\definecolor{steal}{RGB}{0, 170, 170}
\definecolor{smagenta}{RGB}{170, 0, 102}
\definecolor{sred}{RGB}{204, 0, 34}
\definecolor{dspurple}{RGB}{85, 0, 128}
\definecolor{dsred}{RGB}{128, 0, 21}
\definecolor{dsblue}{RGB}{0, 51, 128}
\definecolor{dsteal}{HTML}{02484D}

\newcommand{\ceil}[1]{\left\lceil #1 \right\rceil}
\newcommand{\floor}[1]{\left\lfloor #1 \right\rfloor}

\theoremstyle{plain}
\newtheorem{theorem}{Theorem}
\newtheorem{lemma}{Lemma}
\newtheorem{proposition}{Proposition}

\theoremstyle{definition}
\newtheorem{definition}{Definition}

\theoremstyle{remark}
\newtheorem{remark}{Remark}
\newtheorem{example}{Example}

\usepackage{titlesec}
\titleformat{\section}[block]{\filcenter\bfseries}{\thesection}{1em}{}
\titleformat{\subsection}[runin]{\normalfont\bfseries}{\thesubsection.}{0.5em}{}[.]

\renewenvironment{abstract}
  {\begin{center}%
     \begin{minipage}{0.8\textwidth} 
     \small                         
     \noindent\textbf{Abstract:}    
  }
  {\end{minipage}\end{center}}
  
\usepackage{fancyhdr}
\title{The height function of a sparse collection: \\
a Bellman function approach}
\author{
  Shivam Aggarwal, $\:$
  Samuel Hernandez, $\:$
  Irina Holmes Fay, $\:$ and $\:$
  Jennifer Mackenzie
}
\date{} 

\begin{document}

\maketitle

\begin{abstract}
Sparse operators have emerged as a powerful method to extract sharp constants in harmonic analysis inequalities, for example in the context of  bounding singular integral operators. We investigate the level sets of height functions for sparse collections, or, in other words, weak-type (1,1) inequalities for sparse operators applied to constant functions.
We use another notable method from dyadic harmonic analysis,  also famous for its ability to produce sharp constants, the Bellman function method. Specifically, we find the exact Bellman function maximizing level sets of $\mca_\alpha \unit$, where $\mca_\alpha$ is the (localized) sparse operator associated with a binary Carleson sequence.
\end{abstract}


\noindent 
This work began as an undergraduate research project in Spring 2024.  
While the problem under consideration has been chosen for its clarity rather than complexity, this simplicity is a feature: it allows us to present a streamlined exposition of the Bellman function method that is accessible to students encountering it for the first time, while still capturing the essence of the technique.

\section{Introduction}


\noindent Harmonic analysis studies how complicated signals or functions can be understood in terms of simpler, ``building-block'' functions. The structure and properties of such building-blocks depend on the context: what sort of signals or functions are we interested in? What about them are we measuring? What special properties would make their analysis simpler, computations clearer, and problems more tractable? For example, classical Fourier analysis famously uses combinations of sine and cosine functions of various frequencies as its fundamental ``building blocks.''

\vspace{0.1in}
This paper is focused on dyadic harmonic analysis, where building-blocks involve 
 indicator functions 
    $
    \unit_J(t),
    $
 where $J$ is a \textit{dyadic interval} -- from the Greek word $\delta\acute{\upsilon}\alpha\varsigma$ (\textit{dyas}), meaning ``pair'' or ``two,'' reflecting the repeated halving of intervals.
 Dyadic methods have profoundly influenced modern analysis, answering deep questions by translating difficult continuous problems into manageable combinatorial or discrete ones. A celebrated example is Stefanie Petermichl's groundbreaking work \cite{Petermichl2000}, where she recast the classical Hilbert transform -- a fundamental operator extensively studied in analysis -- in terms of simpler dyadic shift operators, revealing entirely new insights, and dramatically simplifying its analysis in weighted settings.

\vspace{0.1in}

Dyadic operators, roughly speaking, map a function $f(t)$ to another function, defined as a sum over dyadic intervals $J$ of various terms involving $\La f\Ra_J \unit_J(t)$, where for any interval $J$ with Lebesgue measure (length) $|J|$,
    $$
    \La f\Ra_J := \frac{1}{|J|}\int_J\: f(t)\,dt
    $$
denotes the average of $f$ over $J$. 
As we detail in Section \ref{S:2}, \textit{sparse operators} involve summing only over a very special \textit{subcollection} of dyadic intervals, namely a \textit{sparse collection}. These are frequently thought of as ``the next best thing'' to a \textit{pairwise disjoint} collection: while overlap is allowed, we have uniform control over the amount of overlap. 

\vspace{0.1in}

We investigate the maximal possible size of the ``sparse generations'' making up such a collection (see Sections \ref{Ss:SparseGeneration} and \ref{Ss:MainQuestion}), a problem which can also be framed as a question about the level sets of sparse operators (see Section \ref{Ss:LevelSets}).
To answer this question, we appeal to the ``Bellman function'' of the problem. Originating in control theory, this method was brought to harmonic analysis in the works  \cite{NTV1999,NTV2001,Volberg2001BellmanApproach}, and has since established itself as a powerful and beautiful method. We refer the reader to \cite{VVbook} for a comprehensive resource on this topic. 

 In Section \ref{S:Bellman}, we explain how to form the Bellman function of the problem, and extract its main properties. In Section \ref{S:Supersolutions}, we focus in on two of these properties: the Main Inequality and the Obstacle Condition. This is the crucial turning point: we will see that our Bellman function is actually the \textit{smallest} function satisfying these two properties. Known as the ``Least Supersolution'' property of Bellman functions, it translates the original harmonic analysis problem into an optimization problem: find the smallest function with these properties.

 In Section \ref{S:Construct}, we construct a function $\bfg$  which minimizes the family of supersolutions, making it a \textit{candidate} for the Bellman function. See \eqref{eq:COMPLETE CANDIDATE} for the full formula for our candidate. Finally, we prove in Section \ref{sec:true candidate} that our candidate itself is a supersolution, completing the proof that $\bfg$ is indeed the exact Bellman function of the problem.


\section{Background and Notations}
\label{S:2}
\subsection{Dyadic Intervals}

\begin{wrapfigure}{r}{0.5\textwidth} 
  \vspace{-\baselineskip}              
  \centering
  \includegraphics[width=\linewidth]{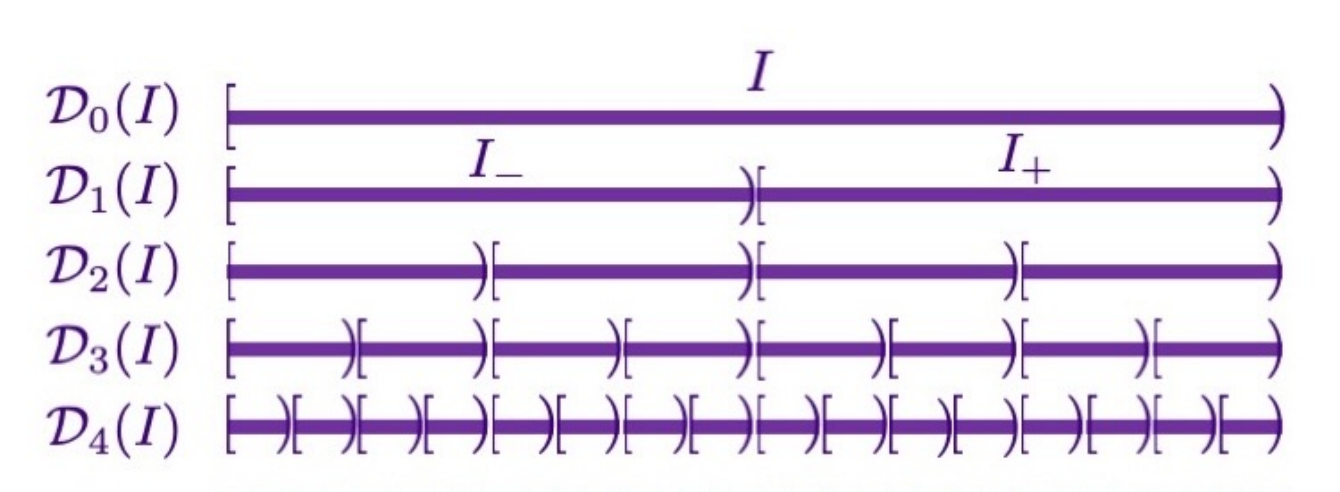}
  \caption{The first five dyadic generations in $\mcd(I)$.}
  \label{fig:DofI}
\end{wrapfigure}

Let $I=[a,b)$ be a real interval, and let $I_-$ and $I_+$ denote the left and right halves of $I$, respectively:
    $$
    I_-=\big[a, \: \tfrac{a+b}{2}\big); \:\:\: I_+=\big[\tfrac{a+b}{2}, \: b\big).
    $$
The dyadic grid adapted to $I$, denoted $\mcd(I)$, is the collection of  subintervals of $I$ defined recursively as follows (see Figure \ref{fig:DofI}):
    \begin{itemize}
        \item $\mcd_0(I) := \{I\}$ $\:\:$ (dyadic generation $0$: the main interval $I$);
        \item $\mcd_1(I)=\{I_-, \: I_+\}$ $\:\:$ (dyadic generation $1$: the two \textit{dyadic children} of the main interval $I$);
        \item Generally, for $k\geq 1$, we define the $k^{\text{th}}$ dyadic generation $\mcd_k(I)$ to be the collection of all dyadic children of intervals in the previous generation $\mcd_{k-1}(I)$, i.e., 
            $
            \mcd_{k}(I) := \{J_-, \: J_+ \: : \: J \in \mcd_{k-1}(I)\}.
            $
    \end{itemize}
Finally, we let
    $$
    \mcd(I) := \bigcup_{k=0}^\infty \mcd_k(I).
    $$

\noindent We note some simple but crucial properties of such collections:
    \begin{itemize}
        \item For every $J\in\mcd(I)$, we have 
            $
            |I|=2^k |J|
            $
        for some integer $k\geq 0$.
        \item Every dyadic generation $\mcd_k(I)$ forms a \textit{partition} of $I$.
        \item If $J, K \in \mcd(I)$, then
            $J \cap K$ is one of $\{\emptyset, J, K\}$.
        In other words, if two dyadic intervals have non-empty intersection, then one must contain the other. This is the quintessential feature of the geometry of dyadic intervals (and dyadic \textit{cubes}, in several dimensions). 
    \end{itemize}

\begin{center}
  {\Large\ding{167}} 
\end{center}

Consider an integrable function $f:I\rightarrow \mbr$. We will associate $f(t)$ with another function, $\mca_\scs f(t)$, also defined on $I$, built from indicators over certain dyadic subintervals $\scs\subset\mcd(I)$, weighted by the corresponding averages of $|f|$:
    \begin{equation}
        \label{e:first}
    f \: \mapsto \: \mca_\scs f(t) := \sum_{K\in\scs} \La |f|\Ra_{K} \unit_K(t); \:\: t\in I.
    \end{equation}

\begin{example} \label{ex:omg}
    Say we take $\scs=\mcd(I)$ above. Since every $t\in I$ is contained in a unique element $I_k(t) \in \mcd_k(I)$ of each dyadic generation, we may express the sum in \eqref{e:first} as:
        $
        \sum_{k=0}^\infty \La |f|\Ra_{I_k(t)}.
        $
    Lebesgue's Differentiation Theorem tells us that, for almost all $t\in I$, $\lim_{k\rightarrow\infty} \La |f|\Ra_{I_k(t)} = |f(t)|$. So, the sum in \eqref{e:first} blows up to $\infty$ for almost all $t\in I$ such that $f(t)\neq0$.
\end{example}

\begin{example} \label{ex:disjoint}
    At the other extreme, assume now that $\scs$ is a \textit{disjoint collection}, i.e. the intervals $K\in\scs$ are \textit{pairwise disjoint} dyadic subintervals of $I$. In this case, any $t\in I$ is contained in \textit{at most one} element of $\scs$, so $\mca_\scs f(t)$ is certainly finite everywhere.
\end{example}

\begin{example} \label{ex:leftmost}
    Suppose $I=[0,1)$ and we consider the collection $\scs$ made up of the leftmost interval in each generation $\mcd_k[0,1)$:
        $
        \scs = \{ [0, 2^{-k}) \: :\: k \in \mbz_{\geq 0}\}.
        $
    Then, note that there is exactly one point in $I$ which is contained in infinitely many elements of $\scs$, namely $\{0\}$.
    So, in this case, the sum on the right-hand side of \eqref{e:first} is finite \textit{almost everywhere}.
\end{example}

The main factor determining the outcome of all three examples above is the same: the amount of overlap between the elements of $\scs$. Example \ref{ex:leftmost} suggests there is a ``reasonable'' middle-ground between the ideal situation of disjointness, and the untractable overlap in Example \ref{ex:omg}. We describe this precisely next.


\subsection{Carleson Sequences} 
    Let $I=[a,b)$ be a real interval equipped with its dyadic grid $\mcd(I)$. 
    We work with binary sequences $\alpha$ ``adapted'' to $I$, i.e. indexed by the dyadic subintervals of $I$:
$$\alpha = \{\alpha_J\}_{J\in\mcd(I)}, \:\: \alpha_J \in \{0,1\}, \text{ for all } J \in \mcd(I).$$
    We can also think of $\alpha$ as a selection procedure: define
    $$
    \scs_\alpha := \{K\in \mcd(I): \: \alpha_K = 1\},
    $$
the collection of dyadic subintervals of $I$ ``selected'' by $\alpha$.

For every $J\in\mcd(I)$, define the quantity:
    $$
    A(\alpha; \: J) := \frac{1}{|J|}\sum_{K\in\mcd(J)}\alpha_K |K| = \frac{1}{|J|}
    \sum_{\substack{K\in\scs_\alpha:\: K \subseteq J}} |K|.
    $$
This is an averaging procedure for $\alpha$ over each interval $J$, relating the total size of selected intervals contained in $J$, to the size of $J$. 
Note that, if $J\in\scs_\alpha$, then $A(\alpha; \: J) \geq 1$, with equality if and only if $\alpha_K=0$ for all $K\subsetneq I$ (in other words, if and only if $\alpha$ selects no further intervals below $J$). On the other hand, $A(\alpha; \: J)>1$ means there must be \textit{some} amount of overlap ($\alpha$ must select at least one $K\subsetneq J$). In fact, the further away from $1$ this quantity is, the more weight $\alpha$ must pack within $\mcd(J)$. So $A(\alpha; \: J)$  serves as a natural measure for the amount of overlap within each interval. As we describe next, the key is to have \textit{uniform} control over $A(\alpha; \: J)$, for all $J$. 

\begin{definition}
Let $\alpha$ be a binary sequence adapted to $I$ and a constant $C\geq 1$. We say that $\alpha$ is \textbf{$C$-Carleson} if and only if
    $$
    \|\alpha\|_{\text{Car}} := \sup_{J\in\mcd(I)} A(\alpha; \: J) \leq C.
    $$
Then, $\|\alpha\|_{\text{Car}}$ is called the \textbf{Carleson constant} of the sequence $\alpha$. We also say the associated collection $\scs_\alpha$ of $\alpha$-selected intervals is \textbf{$C$-Carleson}.
Let $\mfs_C(I)$ denote the set of all binary, $C$-Carleson sequences adapted to $I$.
\end{definition}

Amazingly, a $C$-Carleson collection $\scs\subset\mcd(I)$ has an equivalent formulation which looks quite different from the Carleson condition: $\scs$ is $C$-Carleson if and only if $\scs$ is ``$\sfrac{1}{C}$-sparse.'' We say a collection is \textbf{$\eta$-sparse} if we can associate every $K\in\scs$ with a measurable subset $E_K\subset K$, such that $|E_K|\geq\eta|K|$, and the sets $\{E_K\}_{K\in\scs}$ are \textit{pairwise disjoint}. This formulation is often used to prove boundedness of sparse operators, and reinforces likening sparseness to the next best thing to disjointness. We refer the reader to \cite{LNBook} for a proof.


\begin{center}
  {\Large\ding{167}} 
\end{center}

 Remark that $\|\alpha\|_{\text{Car}}\geq 1$ for all $\alpha \not\equiv 0$, and
    $
    \|\alpha\|_{\text{Car}} = 1 \text{ if and only if } \scs_\alpha
    \text{ is a disjoint collection}.
    $
Moreover, it is enough to check the Carleson condition only on $\alpha$-selected intervals $K\in\scs_\alpha$, that is,
    \begin{equation}
        \label{e:enough}
        \|\alpha\|_{\text{Car}} = c :=  \sup_{K\in\scs_\alpha} A(\alpha; \: K).
    \end{equation}
This is a standard fact, but we include the proof here because it gives us a chance to illustrate a fundamental idea. 
Define for every $J\in\mcd(I)$ the collection of ``$\alpha$-children:''
    \begin{equation*}
    \texttt{ch}_{\alpha}(J) : = \{\text{maximal intervals } K \in \scs_\alpha \text{ such that } K \subsetneq J\},
    \end{equation*}
where \textit{maximality} is with respect to set inclusion; specifically, $K\in\scs_\alpha$ is selected  for $\texttt{ch}_{\alpha}(J)$ if $K\subsetneq J$ \textit{and} $\alpha_L=0$ for any $L\in\mcd(I)$ with $K\subsetneq L\subsetneq J$.

This sort of construction, selecting maximal dyadic intervals which satisfy some property $(P)$, is the key step in the Calder\'{o}n-Zygmund decomposition of functions, in Whitney-type decompositions, covering lemmas, martingale decompositions (see \cite{Grafakos}), and in many ``sparse domination'' arguments (more on this later). It is such a powerful tool across so many types of arguments in large part because \textit{any subcollection built this way will be a disjoint collection} (if non-empty). Consider for instance the collection $\texttt{ch}_\alpha(J)$ defined above. Assuming it is non-empty, let $K, L \in \texttt{ch}_{\alpha}(J)$ with $K\cap L \neq \emptyset$. If $K\neq L$, then one is strictly contained in the other, say $K \subsetneq L$. But this \textit{contradicts the maximality of $K$ in $\texttt{ch}_{\alpha}(J)$}: the dyadic interval $L$ also satisfies property $(P)$, and yet it strictly contains $K$, meaning $K$ cannot be maximal.

Returning to the proof of \eqref{e:enough}, it is clear that $c \leq \|\alpha\|_{\text{Car}}$. To see the converse, let $J\in\mcd(I)\setminus \scs_\alpha$, or $\alpha_J=0$. Then
    $$
    A(\alpha; \: J) = \frac{1}{|J|} \sum_{K\in \texttt{ch}_{\alpha}(J)} \: \sum_{L\in\mcd(K)} \alpha_L |L| 
    \leq \frac{1}{|J|} \sum_{K\in \texttt{ch}_{\alpha}(J)} c \cdot |K| \leq c,
    $$
where the last inequality follows exactly because $\texttt{ch}_{\alpha}(J)$ is a pairwise disjoint collection.

\subsection{Sparse Operators}\label{subsec: sparse} 
If $\alpha \in \mfs_C(I)$, the operator discussed in \eqref{e:first} is called a \textbf{sparse operator} associated with $\alpha$:
    $$
    \mca_\alpha f(t) := \sum_{K\in\scs} \La |f|\Ra_{K} \unit_K(t); \:\: t\in I.
    $$
These operators have been the object of intense study in recent years, due to the demonstrated power of so-called ``sparse domination'' arguments. Roughly speaking, suppose $T$ is some difficult operator we wish to show is bounded $T:L^2\rightarrow L^2$, for example. Now suppose further that, for every function $f\in L^2$, we can construct a $C$-Carleson sparse collection $\scs$ such that 
$|Tf(x)| \leq \mca_{\scs}f(x)$ for almost all $x$. If we have an $L^2$-bound on sparse operators arising from $\alpha\in\mfs_C(I)$, the same bound will transfer to $T$. Such domination arguments have already been established for many important classes of singular integral operators. Further fueling the interest is the fact that such ``domination by sparse'' arguments tend to yield \textit{sharp} bounds, suggesting that sparse operators are the perfect ``toy models'' still powerful enough to capture singularities. We refer the reader to \cite{PereyraRev} and the references therein for an excellent survey of these methods. 

\subsection{The height function}

We say $\mca_\alpha$ is \textit{weak-type $(1,1)$ bounded} if and only if there is a constant $c>0$ such that
    $$
    |\{t\in I: \mca_\alpha f(t)\geq \lambda\}| \leq \frac{c}{\lambda} \int_I |f(t)|\,dt, \text{ for all } \lambda>0 \text{ and } f\in L^1(I)
    \text{ with } \|f\|_{L^1}\neq 0.
    $$
Then, the optimal (smallest) such constant $c$ is the norm of the operator $\mca_\alpha$, acting from $L^1(I)$ to the weak Lebesgue space $L^{1,\infty}(I)$. 
This question was recently considered in \cite{HRS1}, for the special case of indicator functions $f=\unit_E$, where $E\subset I$ is a measurable set. The optimal bound in this restricted case was found using the Bellman function method, and served as the inspiration for this undergraduate research project. 

\vspace{0.1in}

In this paper we consider the simpler case where $f \equiv 1$, and we look at maximizing level sets of the function
    \begin{equation*} 
        h_\alpha(t):=  \mca_\alpha \unit (t) = \sum_{J\in\mcd(I)} \alpha_J  \unit_J(t),
    \end{equation*}
also called the \textbf{height function} of the sparse collection $\scs_\alpha$.
Looking more closely, we notice that $h_\alpha(t)$ \textit{counts} the number of elements in $\scs_\alpha$ which contain $t$:
    $$
	 h_\alpha(t) = \#\{K\in\scs_\alpha: t\in K\}.
	$$
The Carleson condition ensures that the set of points $t\in I$ which are contained in infinitely many elements of $\scs_\alpha$ has Lebesgue measure zero (for a proof, see Lemma 2.2 in \cite{IHFValeria}). Therefore, $h_\alpha$ is finite almost everywhere. Moreover, when finite, $h_\alpha$ takes values in the non-negative integers $\mbz_{\geq 0}$.

\subsection{Sparse Generations} 
\label{Ss:SparseGeneration}
Let $\alpha\in \mfs_C(I)$ be a non-zero sequence. Define the collections
    $$
    \mcg_\alpha^0 := \{\text{maximal } K\in\mcd(I) \text{ such that } \alpha_K=1\}, \:\:\: \text{and}
    \:\:\: \mcg_\alpha^m := \bigcup_{K\in\mcg^{m-1}_\alpha} \: \texttt{ch}_{\alpha}(K) \text{ for } m\geq 1.
    $$
Again by maximality, each $\mcg_\alpha^m$ is a \textit{disjoint} collection of intervals in $\mcd(I)$. We denote their union by 
    $$
    S_\alpha^m := \bigcup_{K\in\mcg_\alpha^m}K.
    $$
It is clear from the definition that
    $
    S_\alpha^0 \supseteq S_\alpha^1 \supseteq \ldots \supseteq S_\alpha^m \supseteq \ldots.
    $
    
\begin{example}
    If $\alpha_I=1$, i.e. the main interval $I$ is itself in the collection $\scs_\alpha$, then 
    the first sparse generation contains only one element, $\mcg_0 = \{I\}$, and $S_\alpha^0=I$. Otherwise, if $\alpha_I\neq 0$, then $\mcg_0$ is a disjoint collection of dyadic intervals in $\mcd(I)$ with union $S_\alpha^0 \subseteq I$.
\end{example}


\subsection{Main Question}
\label{Ss:MainQuestion}
Suppose we ask: given $\alpha \in \mfs_C(I)$, how large can the $m^{\text{th}}$ sparse generation $\mcg_\alpha^{m-1}$ be, relative to the main interval $I$? Now, recall that $A(\alpha; \: I)$
measures the total size of $\alpha$-selected intervals, relative to $|I|$. This can be any number in $[0,C]$, so the answer to our question really hinges on ``how much'' we have to work with: for example, if $A(\alpha; \: I)=5$, we should be able to obtain a larger quantity than if $A(\alpha; \: I)=1$.
So the real question is:
    \begin{center}
        How large can $\frac{|S_\alpha^{m-1}|}{|I|}$ be, for $\alpha\in\mfs_C(I)$ with $A(\alpha; \: I)=A$?
    \end{center}
We will answer this by finding the exact ``Bellman function'' of this problem, starting in Section \ref{S:Bellman}. One last detour before that: we re-frame this question in the language of level sets.

\subsection{Level sets}
\label{Ss:LevelSets}
    
For a Carleson sequence $\alpha\in\mfs_C(I)$ and $\lambda\in\mbr$, define:
    $$
    V_\lambda(\alpha) := \frac{1}{|I|}|\{t \in I: h_\alpha(t) \geq \lambda\}|.
    $$
It is easy to see that
    \begin{equation}
        \label{e:levelsets1}
        0 \leq V_\lambda(\alpha) \leq 1, \text{ for all } \alpha \in \mfs_C(I) \text{ and } \lambda \in \mbr,
    \end{equation}
and
    \begin{equation}
        \label{e:levelsets2}
        V_\lambda(\alpha)=1, \text{ for all } \lambda\leq 0.
    \end{equation}

\begin{remark} \label{rem:negativeLambda}
The reason we allow $\lambda<0$ here is a technical one, and will become apparent later. Essentially, $\lambda$ will be one of the variables of our Bellman function, and allowing negative $\lambda$ will simplify the proof of the so-called ``Bellman induction'' (Theorem \ref{Thm:LSP}). More broadly, the entire $\lambda \leq 0$ case will be swept up in the ``Obstacle Condition'' (see Section \ref{ss:OC}).    
\end{remark}

\begin{center}
  {\Large\ding{167}} 
\end{center}

For $\lambda > 0$, note that, for example, the set $\{t\in I: h_\alpha(t)\geq 3.2\}$ is really  $\{t \in I: h_\alpha(t) \geq 4\}$. In general,
    \begin{equation}
        \label{e:V-ceiling}
    V_{\lambda}(\alpha) = V_{\lceil\lambda\rceil}(\alpha),
    \text{ for all } \alpha\in\mfs_C(I) \text{ and } \lambda>0,
    \end{equation}
where $\lceil\cdot\rceil$ denotes the ceiling function (the smallest integer greater than or equal to $\lambda$).
So it is enough to focus on level sets $V_m(\alpha)$ for positive integers $m$. But observe that these return precisely 
\textit{the size of the $\alpha$-sparse generations}, relative to the size of main interval $I$:
    $$
    V_0(\alpha) = 1; \:\:\:\: 
    V_1(\alpha) = \frac{|S_\alpha^0|}{|I|}; \:\:\:\:
    V_2(\alpha) = \frac{|S_\alpha^1|}{|I|}; \:\ldots\:
    V_m(\alpha) = \frac{|S_\alpha^{m-1}|}{|I|}; \ldots
    $$
For $m\in\mbn$, the level set $\{t \in I: h_\alpha(t) \geq m\}$ asks: which $t\in I$ are contained in at least $m$ elements of $\scs_\alpha$? But elements of $\scs_\alpha$ are \textit{dyadic intervals}, so for example if $t$ is contained in two distinct elements $K_1, K_2$ of $\scs_\alpha$, then either $K_1\subsetneq K_2$ or $K_2\subsetneq K_1$ must hold. In other words, $h_\alpha(t) \geq m$ if and only if there are $m$ elements $K_0,K_1,\ldots K_{m-1}\in\scs_\alpha$ such that 
    $$
    t \in K_{m-1} \subsetneq \ldots \subsetneq K_1 \subsetneq K_0.
    $$
To be contained in an element of $\scs_\alpha$ automatically means being contained in $S_\alpha^0$ (the ``support'' of $\alpha$ in a sense), so if we want those $t\in I$ contained in \textit{at least one} element of $\scs_\alpha$, this is exactly $S_\alpha^0$. Generally, which $t\in I$ are contained in 
\textit{at least} $m$ elements of $\scs_\alpha$? Precisely those $t\in S_\alpha^{m-1}$.

\section{Bellman Function: Definition and Properties} \label{S:Bellman}

\begin{definition}\label{def:bellman-function}
    Fix $C\geq 1$ and define, for real numbers $A, \lambda\in\mbr$:
	\begin{equation}
	\label{e:Bdef}
	\mbg_C(A,\lambda) := \sup_{\alpha} \:V_\lambda(\alpha),
	\end{equation}
where supremum is over all all $\alpha \in \mfs_C(I)$ with fixed average
$A(\alpha; \: I) = A$. We say any such $\alpha$ is an \textbf{admissible sequence} for $\mbg_C(A,\lambda)$.
\end{definition}

We want to consider only points $(A,\lambda)$ for which there exists an admissible sequence (otherwise, we are taking supremum over the empty set). Clearly, any such point will satisfy $A\in[0,C]$. The following lemma ensures that the converse also holds.

    \begin{proposition}
        Let $C \ge 1$. Then, for every $A \in [0,C]$, there exists a binary $C$-Carleson sequence $\alpha$ such that $A(\alpha; I) = A$.
    \end{proposition}

\begin{proof}
We construct admissible sequences according to the following cases.
    \begin{enumerate}[label=(\roman*)]
        \item Suppose $A = 0$. In this case, the \textit{only} admissible sequence is $\alpha \equiv 0$, that is, $\alpha_J = 0, \forall J \in \mcd(I)$.

        \item Now, suppose $A \in (0, 1)$. Then, we obtain a simple $\alpha$ directly from the binary expansion of $A$:
            $$
                A = \sum_{m=1}^\infty \frac{a_m}{2^m} = \frac{a_1}{2} + \frac{a_2}{2^2} + \cdots,
            $$
        where every $a_m \in \{0, 1\}$.
        We can then construct a one-generation sequence $\alpha$. First, let $\alpha_I = 0$, as we cannot have $\alpha_I = 1$ when $A<1$. Next,
        \begin{itemize}
            \item If $a_1 = 1$, we select $I_+$ and proceed to look within $I_-$ for our next interval: if $a_2 = 1$, select $I_{-+}$ and proceed to look within $I_{--}$; otherwise, if $a_2 = 0$, continue to look for the next interval within $I_{-+}$, etc.
            \item Similarly, if $a_1 = 0$, rather than selecting $I_+$, we would have continued to look within $I_+$, and chosen $I_{++}$ if $a_2 = 1$,  and continued to look within $I_{+-}$. If $a_2 = 0$, we would have continued to look within $I_{++}$, etc.
            \end{itemize}
        In this way, we essentially construct an interval of measure $A$, relative to $|I|$ (see Figure \ref{fig: (0,1) admissible sequence}).

        \begin{figure}[ht]
            \centering
            \includegraphics[width=0.45\linewidth]{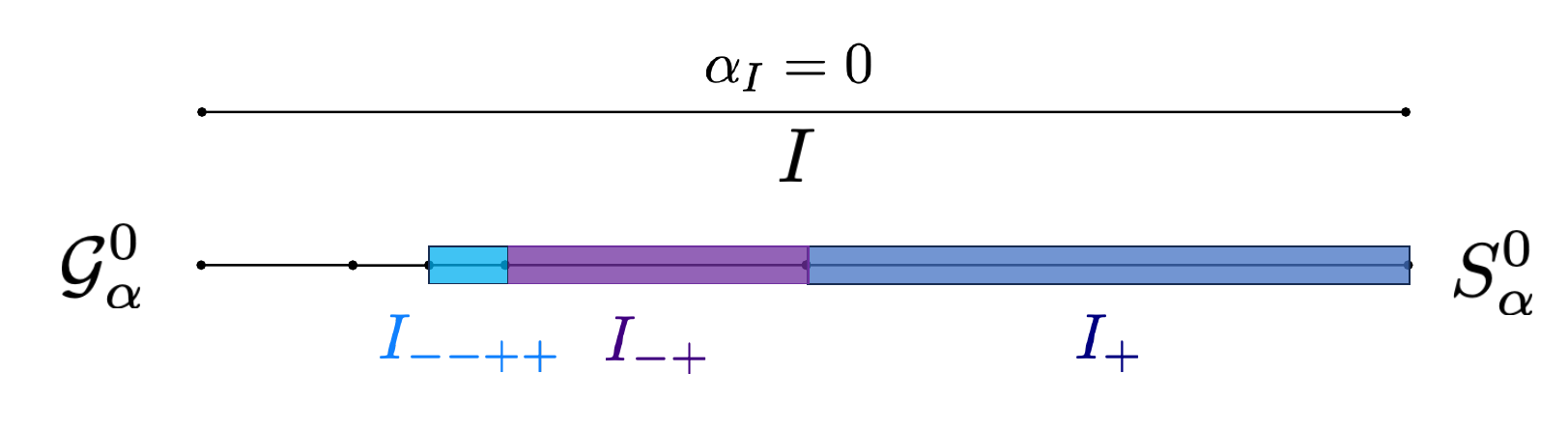}
            \caption{When $A = \sfrac12 + \sfrac14 + \sfrac1{16}$, we first select $I_+$ then continue to look within the left child interval, $I_-$.}
            \label{fig: (0,1) admissible sequence}
        \end{figure}
        
        \item If $A=1$, we can trivially select $\alpha_I = 1$. Alternatively, we can let $\alpha_I = 0$ and let $\alpha_J = 1$ for intervals $J$ that form a partition of $I$. See Figure \ref{fig: nontrivial A=1 sequence}.

        \begin{figure}[ht]
        \centering
        \begin{minipage}[c]{0.45\linewidth}
            \centering
            \includegraphics[width=.75\linewidth]{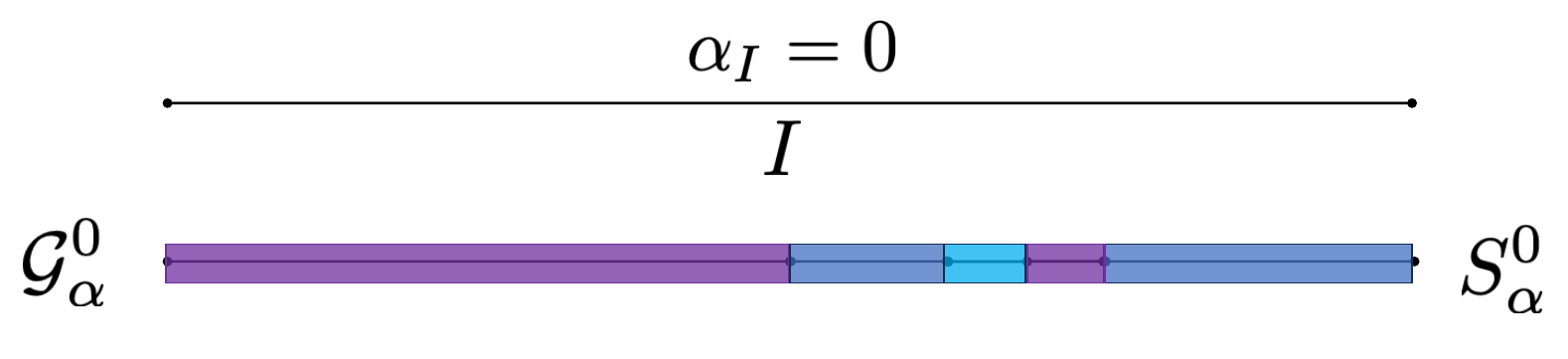}
            \caption{A nontrivial construction for the case $A=1$.}
            \label{fig: nontrivial A=1 sequence}
        \end{minipage}
        \hspace{7pt}
        \begin{minipage}[c]{0.45\linewidth}
            \centering
            \includegraphics[width=.75\linewidth]{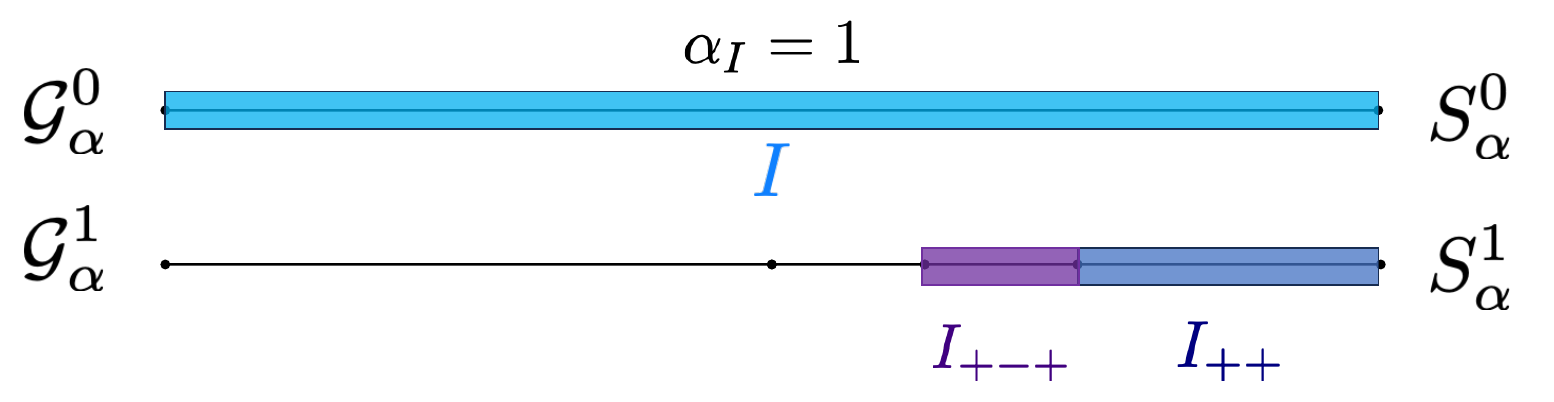}
            \caption{When $A = 1 + \sfrac14 + \sfrac18$, we put a ``roof'' over the construction for the case $A = \sfrac14 + \sfrac18$.}
            \label{fig: (1,2) admissible sequence}
        \end{minipage}
        \end{figure}

        \item If $A > 1$, we write $A = \floor{A} + \{A\}$, where $\{A\}$ denotes the fractional part of $A$ and $\floor{\cdot}$ is the floor function.
        \begin{enumerate}
            \item When $\{ A \} = 0$, we can let $\alpha_J = 1$ for every interval $J$ in generations $\mcd_0(I)$ through $\mcd_{A - 1}(I)$, and $\alpha_J=0$ for all other $J\in\mcd(I)$.

            \item When $\{ A \} \neq 0$, first let $\alpha_J = 1$ for every interval $J$ in generations $\mcd_0(I)$ through $\mcd_{\lfloor A \rfloor - 1}(I)$. Then, for the remaining $\{A\}$, we adapt the admissible sequence from the case $A \in (0,1)$ to the intervals in  generation $\mcd_{\lfloor A \rfloor - 1}(I)$. That is, we select $\floor{A}$ generations of intervals as a ``roof'' for the $A \in (0,1)$ case (See Figure \ref{fig: (1,2) admissible sequence}).
        \end{enumerate}
    \end{enumerate}

\end{proof}

 We can now state the domain and range of our Bellman function:
    $$
    \mbg_C : \Omega_C \rightarrow [0,1],
    $$
where
    $
    \Omega_C:=[0,C]\times\mbr,
    $
and the range comes from \eqref{e:levelsets1}. Also observe that \eqref{e:V-ceiling} translates to
    \begin{equation*}
        \mbg_C(A,\lambda) = \mbg_C(A, \lceil\lambda\rceil).
    \end{equation*}

\subsection{Bellman Function Properties} 
\subsubsection{Independence from the main interval} The careful reader may have remarked that we did not write, for example, $\mbg_C^I$, i.e. we did not keep track of which real interval $I=[a,b)$ is exactly. This is because  $\mbg_C(A,\lambda)$ \textit{is the same no matter which main interval we choose to work within}. 
This is a standard but essential feature of problems amenable to the Bellman approach, and we will see later that it plays a foundational role in obtaining the so-called ``Main Inequality'' of the problem.

This property is especially easy to see in our case because $h_\alpha$ is intrinsic to $\alpha$ alone, so (unlike in \cite{HRS1}, for example) the only objects in our problem are sequences $\alpha \in \mfs_C(I)$, no other functions, weights, etc. Suppose $\alpha \in \mfs_C(I)$ has average $A(\alpha; \: I)=A$, and let $I'=[c,d)$ be any other real interval. The binary tree structure assigning every $J\in \mcd(I)$ an $\alpha_J \in \{0,1\}$ can be ``copy-pasted'' from $I$ to any other interval in the obvious way: to every $J'\in\mcd_k(I')$, assign the same $\beta_{J'} := \alpha_J$ as the corresponding $J\in \mcd_k(I)$. Then $A(\beta; \: I')=A(\alpha; \: I)$ and the sparse generational structure is preserved, so the size of the level set $\{t\in I': h_\beta(t) \geq \lambda\}$ relative to $|I'|$, is the same as the size of the level set $\{t\in I: h_\alpha(t)\geq \lambda\}$ relative to $|I|$. In other words, $\mbg_C(A,\lambda)$ takes supremum over the same set of real numbers regardless of $I$.

\subsubsection{Boundary Values and the Obstacle Condition}
\label{ss:OC}

We can find $\mbg_C$ exactly for certain extreme situations. For example, if $A=0$, then the \textit{only} admissible sequence for $\mbg_C(A,\lambda)$ is the zero sequence $\alpha\equiv 0$. In this case, $h_\alpha\equiv 0$ is the identically zero function, so the level set $\{t\in I: \: h_\alpha(t) \geq \lambda\}$ is all of $I$ if $\lambda\leq 0$, and $\emptyset$ otherwise.
So, 
	\begin{align*}
	    \mbg_C(0, \lambda) = 
        \begin{cases}
    		1, & \text{ if } \lambda\leq 0\\
    		0, & \text{ if } \lambda>0. 
	    \end{cases}
	\end{align*}
In fact, \eqref{e:levelsets2} shows that the first part is always true, for all $A\in[0,C]$:
	\begin{equation*}
	\mbg_C(A,\lambda) = 1, \text{ for all } \lambda \leq 0.
	\end{equation*}
We will call this property the \textbf{obstacle condition}. 

\begin{center}
  {\Large\ding{167}} 
\end{center}

\begin{remark} \label{rem:whenCis1}
    Another extreme situation occurs when the parameter $C$ takes the value $1$. In this case, we can find $\mbg_1$ directly: $C=1$ forces any element of $\scs_\alpha$ to have no $\alpha$-children, so $\scs_\alpha$ must be a collection of \textit{pairwise disjoint} dyadic subintervals of $I$. In other words, the sparse generational structure of $\scs_\alpha$ must be of the form
    $$
    \mcg_\alpha^0 = \{K_n\}_n \:\text{ and }\: \mcg_\alpha^m=\emptyset \text{ for all } m \geq 1,
    $$
where $\{K_n\}_n\subset\mcd(I)$ are (finitely or countably many) disjoint subintervals which satisfy
	$$
	\frac{1}{|I|}\sum_n |K_n| = \frac{|S_\alpha^0|}{|I|} =A.
	$$
Then every $t\in I$ is contained in at most one element of $\scs_\alpha$, and we have
    \begin{equation}
        \label{e:Cis1Bellman}
        \mbg_1(A,\lambda) = \begin{cases}
		1, & \text{ if } \lambda\leq 0\\
		A, & \text{ if } 0<\lambda\leq 1\\
		0, & \text{ if } \lambda>1.
	\end{cases}
    \end{equation}
\end{remark}


\subsubsection{The Main Inequality} Let $A_1, A_2 \in [0,C]$ and $\lambda \in \mbr$. Since $\mbg_C$ is \textit{independent of the choice of main interval}, we are free to think of  $\mbg_C(A_1, \lambda)$ and $\mbg_C(A_2, \lambda)$ as \textit{two separate problems}, which we can each consider on any real interval we wish. So suppose we have two intervals $I_1$ and $I_2$, and we take for each $i\in\{1,2\}$:
    $$
    \alpha_i \in \mfs_C(I_i) \text{ with } A(\alpha_i; \: I_i)=A_i, \text{ an admissible sequence for } \mbg_C(A_i, \lambda).
    $$
By definition of supremum, for every $\epsilon>0$, we can further choose these $\alpha_i$'s in such a way that
    \begin{equation}
        \label{e:almost-sup}
            V_\lambda(\alpha_i) = \frac{1}{|I_i|}|\{t \in I_i: \: h_{\alpha_i}(t) \geq \lambda\}| 
            > \mbg_C(A_i,\lambda)-\epsilon.
    \end{equation}

Now, suppose the two main intervals are actually the left and right halves of some interval $I$:
    $
    I_1 = I_- \text{ and } I_2=I_+.
    $
Then, we would like to concatenate $\alpha_1$ (adapted to $I_1$) and $\alpha_2$ (adapted to $I_2$), and obtain a new sequence $\alpha$, adapted to $I$. All that remains to be determined is $\alpha_I$, and then we can define:
    $$
    \alpha := \begin{cases}
        \alpha_I, & \text{ if } J=I\\
        \alpha_i(J), & \text{ if } J \subseteq I_i.
    \end{cases}
    $$
Then, for all $t\in I$,
    $$
    h_\alpha(t) = \alpha_I + h_{\alpha_1}(t) + h_{\alpha_2}(t).
    $$

Option one is to \textit{not} select the new main interval, i.e. to assign $\alpha_I=0$. In this case,
    $$
    A(\alpha; \: I) = A:= \frac{A_1+A_2}{2} =: \La A_i\Ra \: \in [0,C].
    $$
The resulting $\alpha \in \mfs_C(I)$ is therefore admissible for $\mbg_C(A,\lambda)$, and then
    \begin{align*}
    \mbg_C(A,\lambda)  \geq V_\lambda(\alpha) &=
    \frac{1}{|I|}|\{t\in I: h_{\alpha}(t) \geq \lambda\}|    \\
    & = \frac{1}{2|I_1|}|\{t\in I_1: \: h_{\alpha_1}(t) \geq \lambda\}| +
    \frac{1}{2|I_2|}|\{t\in I_2: \: h_{\alpha_2}(t) \geq \lambda\}|\\
    &> \frac{1}{2} \bigg(\mbg_C(A_1,\lambda) + \mbg_C(A_2,\lambda)\bigg) - \epsilon,
    \end{align*}
where the last inequality follows from \eqref{e:almost-sup}. Since this holds for all $\epsilon>0$, we can take the limit as $\epsilon \rightarrow 0+$ above and obtain:
    \begin{equation*}
        \mbg_C(A,\lambda) \geq \frac{1}{2} \sum_{i=1}^2 \mbg_C(A_i, \lambda).
    \end{equation*}
In other words, $\mbg_C(\cdot, \lambda)$ is midpoint concave for every fixed $\lambda$.


The other option is to assign $\alpha_I=1$. Now, this means
$
A(\alpha; \: I) = 1 + A,
$
where $A=\La A_i\Ra$ as before. To obtain a $C$-Carleson sequence adapted to $I$, we must have $A+1 \leq C$. Assuming this, we obtain $\alpha \in \mfs_C(I)$, admissible for $\mbg_C(A+1, \: \lambda+1)$. Then
    \begin{align*}
        \mbg_C(A+1, \: \lambda+1) \geq V_{\lambda+1}(\alpha) 
        & = \frac{1}{|I|} |\{t \in I: \: 1+ h_{\alpha_1}(t) + h_{\alpha_2}(t) \geq \lambda +1\}|\\
        & = \frac{1}{2} \sum_{i=1}^2 \frac{1}{|I_i|}
        |\{t\in I_i: \: h_{\alpha_i}(t) \geq \lambda\}| \\
        & > \frac{1}{2} \bigg(\mbg_C(A_1,\lambda) + \mbg_C(A_2,\lambda)\bigg) - \epsilon.
    \end{align*}
As before, we obtain
    \begin{equation*}
        \mbg_C(A+1, \: \lambda+1) \geq \frac{1}{2} \sum_{i=1}^2 \mbg_C(A_i, \lambda).
    \end{equation*}

\begin{remark}
    The reader may wonder why we started the Main Inequality with two points $(A_i,\lambda)$, with the same second coordinate, as opposed to a more general approach $(A_i,\lambda_i)$. Could we perhaps be missing out on a ``better'' Main Inequality? This is actually a standard feature of weak-type problems: the Bellman function is non-increasing in $\lambda$. To see this in our case, note that $\lambda_1\leq\lambda_2$ implies $V_{\lambda_2}(\alpha) \leq V_{\lambda_1}(\alpha)$ for any $\alpha$, so $\mbg_C(A,\lambda_2)\leq\mbg_C(A,\lambda_1)$, for all $A$. In light of this, the reader can work out as an exercise that, if we start the Main Inequality with two points $(A_i,\lambda_i)$, we end up with a statement equivalent to \eqref{e:mbg-MI} below.
\end{remark}

\begin{center}
  {\Large\ding{167}} 
\end{center}

We summarize the results in this section in the theorem below.

\begin{theorem} \label{T:BellmanProperties}
    Let $C \geq 1$. The Bellman function $\mbg_C$ defined in \eqref{e:Bdef} has the following properties:
    \begin{enumerate}
        \item \textbf{Independence from the Main Interval:} $\mbg_C(A,\lambda)$ is independent of the choice of main interval $I$.
        
        \item \textbf{Domain and Range:} $\mbg_C:\Omega_{C} \rightarrow [0,1]$, where $\Omega_C:=[0,C]\times\mbr$.

        \item \textbf{Ceiling-Invariance in $\lambda$:} For all $(A,\lambda)$ there holds: 
            $
            \mbg_C(A,\lambda) = \mbg_C(A, \lceil\lambda\rceil).
            $

        \item \textbf{Boundary Values:}
        $\mbg_C(0, \lambda) = 0$ for all $\lambda>0$.

        \item \textbf{Obstacle Condition:}
        $\mbg_C(A,\lambda)=1$, for all $\lambda\leq 0$.

        \item \textbf{Main Inequality:}
            \begin{equation}
		\label{e:mbg-MI}
		\mbg_C\bigg(A + \gamma, \: \lambda + \gamma\bigg) \geq \frac{1}{2}\bigg(\mbg_C(A_1, \lambda) + \mbg_C(A_2, \lambda)\bigg),
		\end{equation}
	for all $0\leq A_1, A_2 \leq C$ with $A :=\La A_i\Ra := \tfrac{A_1+A_2}{2}$ and $\gamma\in\{0,1\}$  such that $\gamma+A\leq C$.
        \item \textbf{Monotonicity in $\lambda$:} 
            $$
            \mbg_C(A,\lambda_1) \geq \mbg_C(A,\lambda_2), \text{ for all } \lambda_1 \leq \lambda_2.
            $$
    \end{enumerate}
\end{theorem}

\noindent Two of these properties take the stage going forward: the Obstacle Condition and the Main Inequality. As we see next, $\mbg_C$ is in fact the minimal function with these two properties.

\section{The family of supersolutions}
\label{S:Supersolutions}

\begin{definition}\label{def:supersolution}
    Let $C \geq 1$. We say that a function $G:\Omega_C \rightarrow [0,1]$ is a \textbf{supersolution} if and only if $G$ satisfies the Obstacle Condition:
        \begin{equation}
            \label{e:supsol-OC}
            G(A,\lambda)=1, \text{ for all } \lambda \leq 0,
        \end{equation}
    and the Main Inequality:   
        \begin{equation}
            \label{e:supsol-MI}
           G\big(A + \gamma, \: \lambda + \gamma\big) \geq \frac{1}{2}\bigg(G(A_1, \lambda) + G(A_2, \lambda)\bigg), 
        \end{equation}
    for all $0\leq A_1, A_2 \leq C$ with $A :=\La A_i\Ra := \tfrac{A_1+A_2}{2}$ and $\gamma\in\{0,1\}$  such that $\gamma+A\leq C$.

    \vspace{0.1in}
    Let $\mathscr{G}_C$ denote the collection of all such functions $G$. 
\end{definition}

\begin{theorem}[\textbf{Least Supersolution Property}]
\label{Thm:LSP}
    Let $G\in\scg_C$. Then for all $(A,\lambda) \in \Omega_C$:
        $$
        \mbg_C(A,\lambda) \leq G(A,\lambda).
        $$
\end{theorem}

\begin{proof}
For $\lambda\leq 0$, the result is guaranteed to hold by the Obstacle Condition. So let $(A,\lambda) \in \Omega_C$ with $\lambda>0$, and let $\alpha$ be \textit{any} admissible sequence for $\mbg_C(A,\lambda)$. We will show that
        \begin{equation*}
            V_\lambda(\alpha) \leq G(A,\lambda).
        \end{equation*}
    Then, taking supremum over all admissible $\alpha$, we obtain exactly $\mbg_C(A,\lambda) \leq G(A,\lambda)$.

\vspace{0.1in}

At first sight, it may not be obvious at all how to connect a general class of functions satisfying certain inequalities, with sparse collections. The idea is to ``run the Main Inequality backwards,'' in a process called \textit{Bellman induction}. Supersolution functions, $G$, satisfy the Main Inequality, and we can let $\alpha$ determine the inputs $G$ takes. For every $J\in\mcd(I)$, define:
    $$
    A_J := A(\alpha; \: J); \:\:\:\: \lambda_J := \lambda - \sum_{K\in\mcd(I): K \supsetneq J}\: \alpha_K.
    $$
In particular,
    $$
    A=A_I = \alpha_I + \frac{1}{2}(A_{I_-} + A_{I_+}) \text{ and } \lambda=\lambda_I.
    $$
We let $\alpha_I$ play the role of $\gamma$ in the Main Inequality, and obtain
    $$
    G(A,\lambda) \geq \frac{1}{2} \bigg( G(A_{I_-}, \lambda_{I_-}) + G(A_{I_+}, \lambda_{I_+}) \bigg) = \frac{1}{2} \sum_{J\in\mcd_1(I)} G(A_J, \lambda_J).
    $$
Iterating this procedure, we obtain
    $$
    G(A,\lambda) \geq \frac{1}{2^N} \sum_{J\in\mcd_N(I)} G(A_J, \lambda_J),
    $$
for all $N\in\mbn$.

Suppose first that $\alpha$ is a \textit{finite} sequence, i.e. only finitely many $\alpha_K$'s are $1$. Then there is a dyadic generation $N\in\mbn$ such that $\alpha_J=0$ for all $J\in\mcd_n(I)$ with $n\geq N$, and  $h_\alpha$ is then constant on each terminal interval $J\in\mcd_N(I)$:
    $$
    h_\alpha(t) = \sum_{K\in\mcd(I): K \supsetneq J}\: \alpha_K = \lambda-\lambda_J,
    \text{ for all } t \in J, \: J\in\mcd_N(I).
    $$
We may write then
    $$
    G(A,\lambda) \geq \frac{1}{2^N} \sum_{J\in\mcd_N(I)} G(A_J, \lambda_J)
    \geq \frac{1}{2^N} \sum_{J\in\mcd_N(I): \lambda_J\leq 0} G(A_J, \lambda_J) = 
    \frac{1}{2^N} \cdot \#\{J \in \mcd_N(I): \lambda_J\leq 0\},
    $$
where the Obstacle Condition \eqref{e:supsol-OC} acted as a \textit{stopping condition} to yield the last equality. But now remark that, for $J\in\mcd_N(I)$,
    \begin{align*}
    \lambda_J \leq 0 & \Leftrightarrow \lambda \leq \sum_{K\in\mcd(I): K \supsetneq J}\: \alpha_K =h_\alpha(t), \text{ for all } t \in J\\
        & \Leftrightarrow \lceil\lambda\rceil \leq \sum_{K\in\mcd(I): K \supsetneq J}\: \alpha_K =h_\alpha(t), \text{ for all } t \in J\\
        & \Leftrightarrow J \subset S_\alpha^{\lceil\lambda\rceil-1}
    \end{align*}
where the second equivalence follows because $h_\alpha$ is an \textit{integer} larger than $\lambda$.
Finally,
    $$
    G(A,\lambda) \geq \frac{1}{2^N} \cdot \#\{J \in \mcd_N(I): \lambda_J\leq 0\} =
    \frac{1}{2^N} \cdot \#\left\{J \in \mcd_N(I): J \subset S_\alpha^{\lceil\lambda\rceil-1}\right\} =
    \frac{\left|S_\alpha^{\lceil\lambda\rceil-1}\right|}{|I|}.
    $$
Taking supremum over all admissible $\alpha$'s:
    $$
    G(A, \lambda) \geq \sup_\alpha \frac{\left|S_\alpha^{\lceil\lambda\rceil-1}\right|}{|I|} = \mbg_C(A, \lceil\lambda\rceil) = \mbg_C(A,\lambda).
    $$

Having proved the result for finite sequences $\alpha$, consider now $\alpha\in\mfs_C(I)$ with $A(\alpha; \: I)=A$ (and $\alpha_K=1$ for possibly infinitely many $K$). For every $N\in\mbn$, let $\alpha^{(N)}$ denote the finite sequence obtained by truncating $\alpha$ at dyadic level $N$, i.e.
$$\alpha^{(N)}=\{\alpha^{(N)}_J\}_{J\in\mcd(I)}, \:\:\:
\alpha^{(N)}_J
= \begin{cases}
\alpha_J & \quad \text{if } \frac{|J|}{|I|}>\frac{1}{2^N}, \\
0 & \quad \text{otherwise}.
\end{cases}$$
Then
    $$
    A_N:= A(\alpha^{(N)}; \: I) \text{ satisfies } \lim_{N\rightarrow\infty} A_N=A,
    $$
and
    $$
    h_\alpha(t) = \lim_{N\rightarrow\infty} h_{\alpha^{(N)}}(t) =
    \sup_{N\in\mbn} h_{\alpha^{(N)}}(t).
    $$
Finally,
    $$
    V_\lambda(\alpha) = \lim_{N\rightarrow\infty} \: \frac{1}{|I|}\left|\{t\in I:
    h_{\alpha^{(N)}}(t) \geq \lambda\}\right| 
    \leq \lim_{N\rightarrow\infty} G(A_N, \lambda) = G(A,\lambda),
    $$
where the last equality follows because $G(\cdot, \lambda)$ is continuous.
\end{proof}

\begin{center}
  {\Large\ding{167}} 
\end{center}

\noindent While the Main Inequality in its general form \eqref{e:supsol-MI} was useful for the proof above, it will be easier going forward to observe that it can be reduced to two particular instances.

Let $\gamma=0$, and we obtain for all $(A_i, \lambda) \in \Omega_C$:
    \begin{equation*}
       G(A,\lambda) \geq \frac{1}{2} \sum_{i=1}^2 G(A_i, \lambda), \text{ where } A=\La A_i\Ra.
    \end{equation*}
Since \textit{bounded} midpoint concave functions are continuous (see page 12 in \cite{Donoghue}), we have that $G(\cdot, \lambda)$ is \textbf{concave} for all fixed $\lambda$, for all $G\in\scg_C$. 
If we let $A_1=A_2=A$ and $\gamma=1$ in \eqref{e:supsol-MI}, we obtain the \textbf{Jump Inequality}:
    \begin{equation}
        \label{e:sup-J}
        G(A+1, \lambda+1) \geq G(A,\lambda), \text{ for all } 0\leq A\leq C-1.
    \end{equation}

    \begin{lemma}
        A function $G:\Omega_C\rightarrow [0,1]$ satisfies the Main Inequality \eqref{e:supsol-MI} if and only if $G$ satisfies the following two conditions:
            \begin{enumerate}[label=(\roman*)]
                \item $G(\cdot, \lambda)$ is concave for all $\lambda$;
                \item $G$ satisfies the Jump Inequality \eqref{e:sup-J}. 
            \end{enumerate}
    \end{lemma}

\begin{proof}
    We only need to show that the two conditions imply the more general \eqref{e:supsol-MI}. Condition \textit{(i)} immediately gives \eqref{e:supsol-MI} for $\gamma=0$. The case $\gamma=1$ follows from an application of both jump and concavity:
        $$
        G(A+1, \lambda+1) \geq G(A,\lambda) \geq \frac{1}{2} \sum_{i=1}^2 G(A_i, \lambda),
        $$
    where the first inequality follows from \textit{(ii)}, and the second from \textit{(i)}.
\end{proof}

\begin{center}
 {\Large\ding{167}} 
\end{center}

\subsection{Transition to a minimization problem} The ``true'' Bellman function $\mbg_C$ is itself contained in the collection $\scg_C$, therefore we can now write
    $$
    \mbg_C(A,\lambda) = \min_{G\in\scg_C} \: G(A,\lambda).
    $$
In other words, finding $\mbg_C$ (which originated as a problem in harmonic analysis), is the same as finding \textit{the smallest function in the collection} $\scg_C$. At this point, we can completely detach from the harmonic analysis motivations, and focus on this minimization problem.

\vspace{0.2in}

In the next Section \ref{S:Construct} we will construct, for every $C \geq 1$, a minimizer function $\bfg_C$ such that $\bfg_C \leq G$, for all $G\in\scg_C$.
But remember that the Bellman function $\mbg_C$ is not just any minimizer, it is \textit{the best minimizer}. So our construction really only gives us that $\bfg_C\leq\mbg_C$, in other words that $\bfg_C$ is a \textit{candidate} for $\mbg_C$, but not necessarily equal to $\mbg_C$. 
In order to have equality, $\bfg_C$ \textit{must itself belong to the collection $\scg_C$}! 

For example, the function
    	$$
	f(A,\lambda) = \begin{cases}
	0, & \text{ if } \lambda>0\\
	1, & \text{ if } \lambda\leq 0,
	\end{cases}
	$$
trivially satisfies $f\leq G$ for all $G\in\scg_C$, so it is a minimizer of the collection $\scg_C$. However, it cannot be \textit{the best} minimizer:  if it were, then $f$ would satisfy the Jump Inequality \eqref{e:sup-J}, which would lead to the contradiction
	$
	0=f(1, 1) \geq f(0,0)=1.
	$
    
This is a silly example of a very real danger: when trying such direct constructions, one always runs the risk that ``we could have done better,'' i.e. the candidate is not large enough. For example, maybe it comes out to be a convex function instead of concave, or maybe it fails the Jump Inequality. Either way, it means that one's use of the available ``moves'' to propagate data was not optimal, and there is some ``better'' way to propagate. In problems with several variables, or problems which take place on strange, non-convex domains, this can get very difficult very fast. 

We will prove in Section \ref{sec:true candidate} that our candidate does indeed belong to the collection $\scg_C$, therefore the converse inequality $\mbg_C\leq\bfg_C$ holds, completing the proof that $\mbg_C=\bfg_C$.

\section{Constructing the Candidate}
\label{S:Construct}
Let  $C \geq 1$. Recall the collection $\scg_C$ consists of functions $G:[0,C]\times\mbr\rightarrow[0,1]$ which satisfy:
    \begin{enumerate}
        \item \textit{Obstacle Condition:} $G(A,\lambda)=1$ for all $\lambda\leq 0$.
        \item \textit{Concavity in the first variable:} $G(\cdot,\lambda)$ is concave for all $\lambda$.
        \item \textit{Jump Inequality:} $G(A+1,\lambda+1) \geq G(A,\lambda)$ for all $0\leq A\leq C-1$.
    \end{enumerate}

Our goal is to construct a function $\bfg_C$, also defined on $\Omega_C=[0,C]\times\mbr$, which will satisfy 
    \begin{equation}
        \label{e:MinimizerCondition}
    \bfg_C(A,\lambda) \leq G(A,\lambda), \text{ for all } 
    (A,\lambda) \in \Omega_C \text{ and  all } G \in \scg_C.
    \end{equation}
    Given that we want to construct $\bfg_C$ to be as large as possible, it makes sense to define
    $$\bfg_C(A,\lambda) := 1, \text{ for all } \lambda \leq 0, $$
    for all our candidates.
 The focus therefore will be on constructing $\bfg_C(A,\lambda)$ for $\lambda>0$: we will do this by propagating data through the domain $\Omega_C$, starting with the initial information provided by the Obstacle Condition, 
and then using our two available ``moves:'' horizontal ($A$-direction) concavity, and the Jump Inequality.

\begin{center}
  {\Large\ding{167}} 
\end{center}

For extra clarity, we will first illustrate the construction for some particular values of $C$, before moving on to the general case.

\subsection{\texorpdfstring{The case $C=1$}{TEXT}} Let us see if we can recover the extremal case  in Remark \ref{rem:whenCis1}, using only the assumptions above (and no harmonic analysis). See Figure \ref{fig:C1-build} for an illustration of the process below.

\begin{figure}[h]
\begin{center}
\includegraphics[width=.9\linewidth]{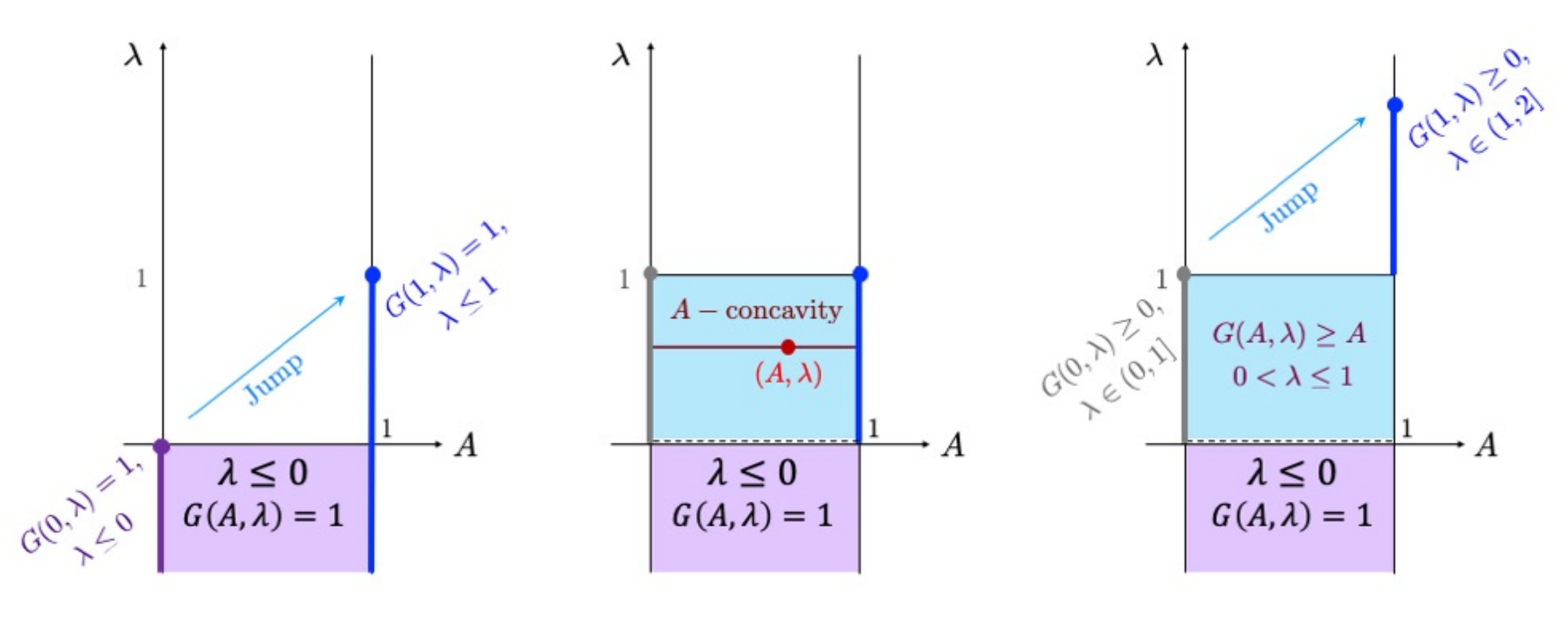}
\end{center}
\caption{Constructing the candidate for $C=1$.}
\label{fig:C1-build}
\end{figure}

Let any $G \in \scg_1$. 
Apply the Jump Inequality to $G(0,\lambda)=1; \: \lambda\leq 0$, and we obtain:
    $
    G(1,\lambda) \geq 1, \text{ for all } \lambda \leq 1.
    $
But since $1$ is the maximum possible value any $G \in \scg_C$ can attain, this is actually an \textit{equality}:
    $$
    G(1,\lambda) = 1, \text{ for all } \lambda \leq 1.
    $$
Now, for each $\lambda \in (0,1]$, we can use $A$-concavity to interpolate between the boundary estimates $G(0,\lambda) \geq 0$ and $G(1, \lambda)=1$: take any $A \in (0,1)$ and write $(A,\lambda)$ as a convex combination
    $$
    (A, \lambda) = (1-\theta) \: (0, \lambda) + \theta \: (1, \lambda),
    $$
for $\theta \in (0,1)$. Then $\theta = A$, and by concavity of $G$ in the first variable:
    $$
    G(A, \lambda) \geq (1-\theta) \: G(0, \lambda) + \theta \: G(1,\lambda) \geq  A.
    $$

Next, we would try to jump the new concavity-data we obtained for $\lambda \in (0,1]$; this step will get more interesting shortly, but for now it only leads us, inductively,  to the trivial estimate:
    $
    G(1,\lambda) \geq  0, \text{ for all } \lambda > 1.
    $

\begin{minipage}[c]{0.4\linewidth}
\begin{center}
\includegraphics[width=.9\linewidth]{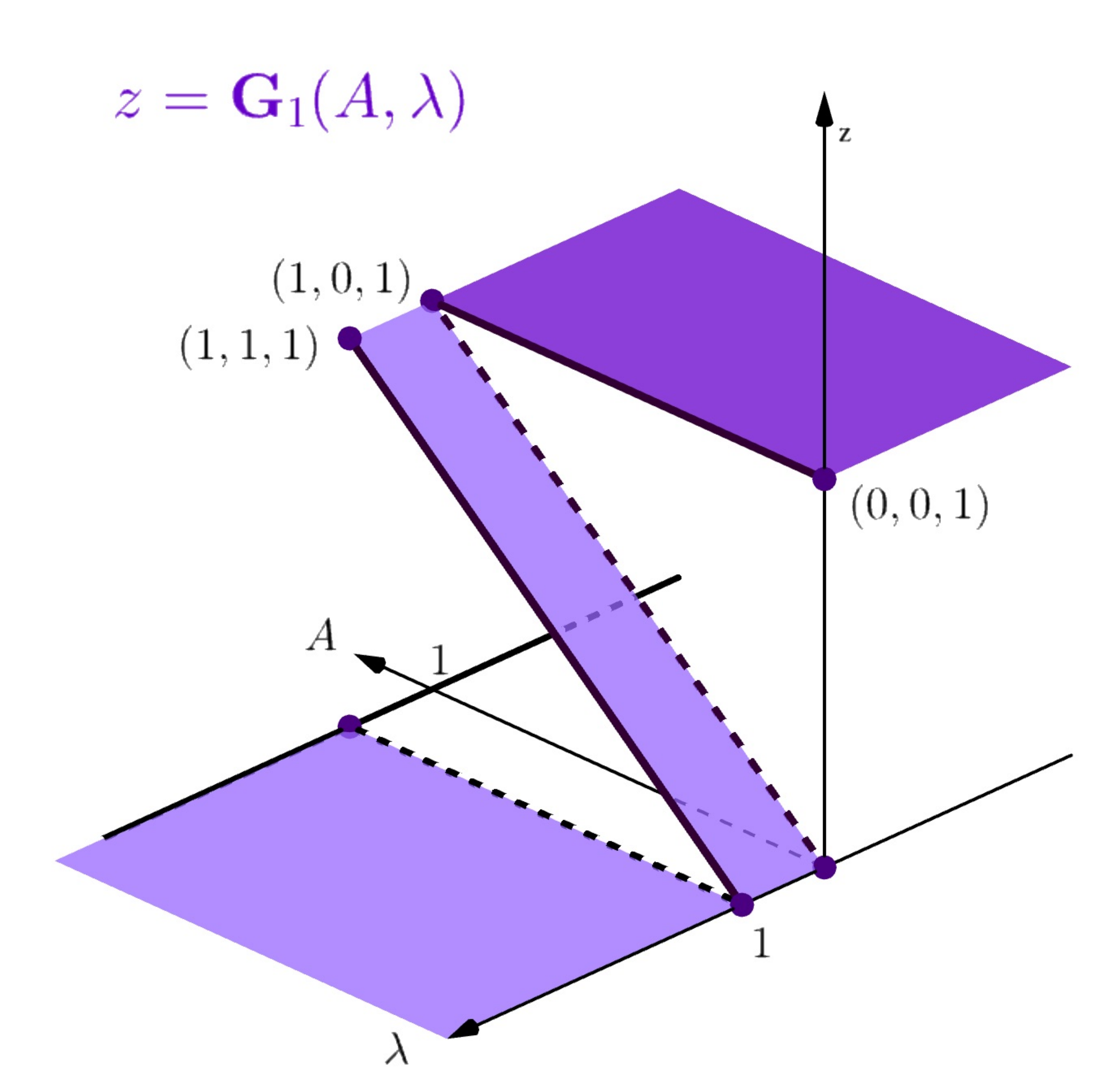}
\end{center}
\end{minipage}
\begin{minipage}[c]{0.6\linewidth}
    Therefore, if we define our candidate as:
    $$
    \bfg_1(A,\lambda) := \begin{cases}
        1, & \text{ if } \lambda\leq 0\\
        A, & \text{ if } \lambda \in (0,1]\\
        0, & \text{ if } \lambda > 1,
    \end{cases}
    $$
we are guaranteed that $\bfg_1$  satisfies \eqref{e:MinimizerCondition}.
    
\vspace{0.1in}

Note, however, that $\bfg_1$ is precisely the function in \eqref{e:Cis1Bellman}, so in this case we secretly already know (from harmonic analysis) that it is the best possible candidate. With caution, we take this as an indication we are on the right track, and using the moves optimally.
\end{minipage}

\begin{remark}
    In this case, we could \textit{only} jump from $A=0$: attempting to use the Jump Inequality on any $A>0$ would take us out of the domain. So, once we are at $\lambda>1$, the only information we can propagate via Jump is the trivial estimate $G(0,\lambda) \geq 0$. 
\end{remark}


\subsection{\texorpdfstring{The case $C=2$}{TEXT}} 
Consider now a function $G\in\scg_2$.  We start, as in the previous case, by ``jumping the obstacle:'' from the Jump Inequality,
    $
    G(1,\lambda) \geq G(0,\lambda-1) = 1 \text{ if } \lambda\leq 1,
    $
so
    $$
    G(1,\lambda)=1 =:\bfg_2(1,\lambda), \text{ for all } \lambda\leq 1.
    $$
But now we can ``jump'' again, this time from $A=1$, where we have just generated new information:
    $
    G(2,\lambda) \geq G(1,\lambda-1)=1 \text{ if } \lambda\leq 2,
    $
so
   $$
    G(2,\lambda)=1 =:\bfg_2(2,\lambda), \text{ for all } \lambda\leq 2.
    $$    

\begin{figure}[h]
    \centering
    \includegraphics[width=\linewidth]{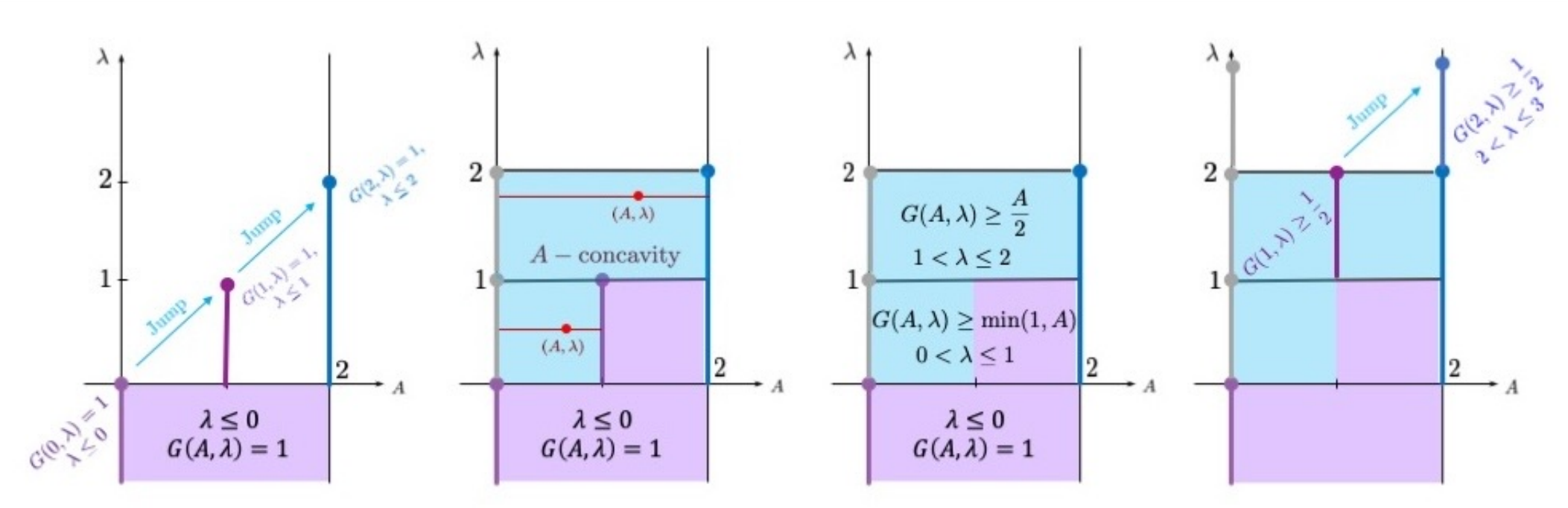}
    \caption{Constructing the candidate for $C=2$.}
    \label{fig:building-C2}
\end{figure}

We have now reached the boundary $A=2$ (see Figure \ref{fig:building-C2}), which means we can no longer jump data. We then use our only remaining move, $A$-concavity, to propagate this new boundary data inwards. Note that we have to wield concavity differently for $\lambda \in (0,1]$ and for $\lambda \in (1,2]$.

If $\lambda\in (0,1]$, then $G$ is maximal at two points: $G(1,\lambda)=G(2,\lambda)=1$. Since $G(\cdot, \lambda)$ is concave, there must then  hold
    $$
    G(A, \lambda)=1=: \bfg_2(A,\lambda), \text{ for all } A\in [1,2], \: \lambda\in (0,1].
    $$
To define $\bfg_2$ for $A\in(0,1)$, we again use concavity of $G$ in the first variable to interpolate between $G(0,\lambda)\geq 0$ and $G(1,\lambda)=1$, and obtain
    $$
    G(A,\lambda) \geq (1-A)\cdot G(0,\lambda) + A \cdot G(1,\lambda) \geq A =: \bfg_2(A,\lambda), 
    \text{ for all } A\in(0,1), \: \lambda\in (0,1].
    $$

For $\lambda \in (1,2]$, we can interpolate directly between $G(0,\lambda)\geq 0$ and $G(2,\lambda)=1$, to obtain via $A$-concavity:
    $$
    G(A,\lambda) \geq \frac{A}{2} =: \bfg_2(A,\lambda), 
    \text{ for all } A\in[0,2] \text{ and } \lambda\in (1,2].
    $$
We now have new data at $A=1$, for $\lambda \in (1,2]$, which we can jump to create new information on the boundary $A=2$:
    $
    G(2,\lambda+1) \geq G(1,\lambda) \geq \frac{1}{2}, \text{ for } \lambda \in (1,2],  
    $
that is
    $$
    G(2,\lambda) \geq \frac{1}{2} =: \bfg_2(2,\lambda), \text{ for all }
    \lambda \in (2,3].
    $$
Now we can again interpolate between the boundaries $A=0$ and $A=2$ for $\lambda \in (2,3]$ and obtain
    $$
    G(A,\lambda) \geq \frac{A}{2} G(2,\lambda) \geq \frac{A}{2^2} =: \bfg_2(A,\lambda),
    \text{ for all } \lambda \in (2,3].
    $$
\begin{wrapfigure}{r}{0.42\textwidth} 
  \vspace{-\baselineskip}              
  \centering
  \includegraphics[width=\linewidth]{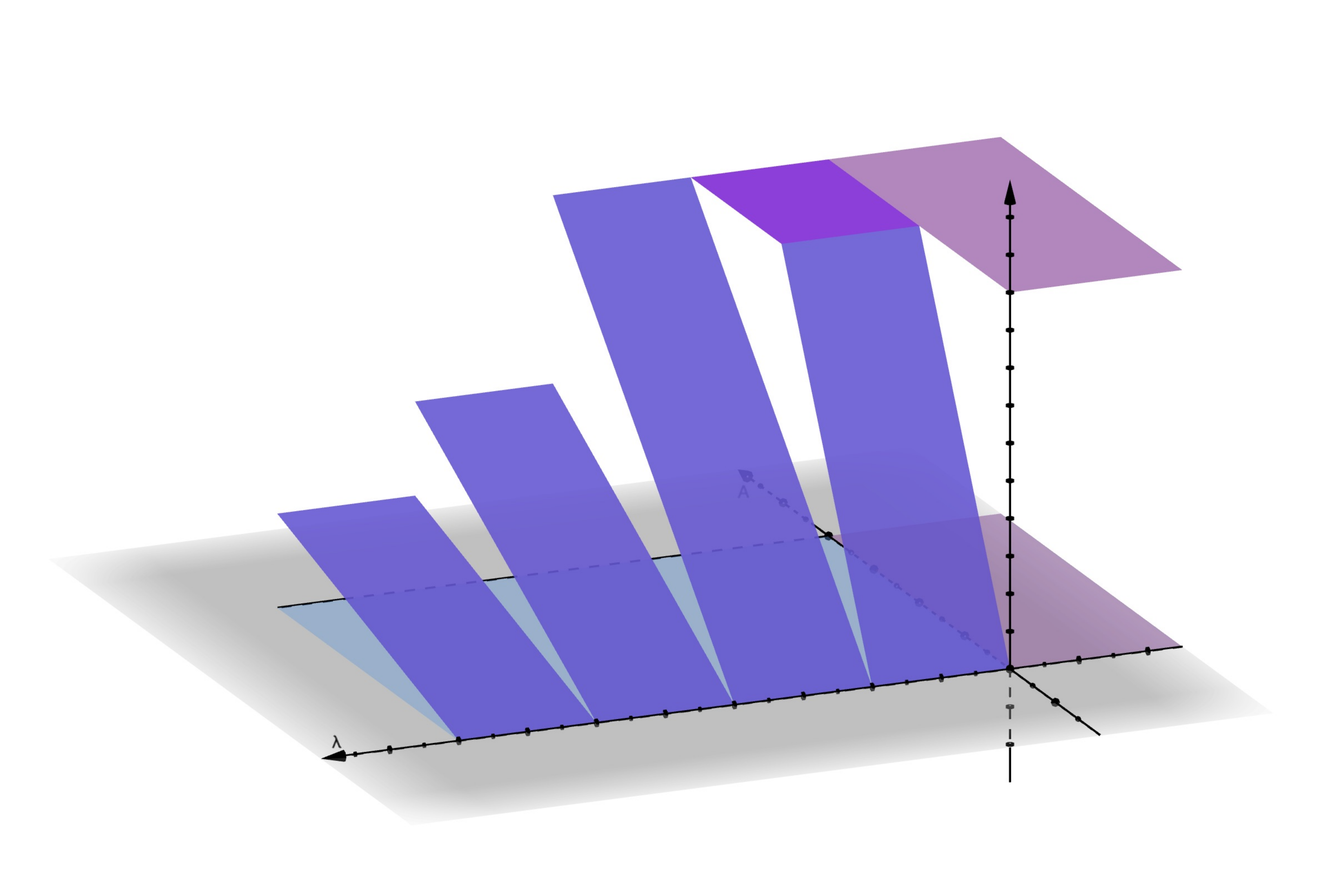}
  \caption{The candidate for $C=2$.}
  \label{fig:C2-built}
\end{wrapfigure}
Note that, in general, for $\lambda>1$, we can essentially reduce the problem to constructing the candidate on the boundary $A=2$. Specifically, let $A\in(0,2)$ and write $A=(1-\theta) \cdot 0 + \theta \cdot 2$. Then $\theta=\sfrac{A}{2}$, and
    $$
    G(A,\lambda) \geq (1-\theta)\cdot G(0,\lambda) + \theta \cdot G(2,\lambda) \geq  
    \frac{A}{2} \cdot G(2,\lambda).
    $$
If we construct $\bfg_2(2,\lambda)$ such that $G(2,\lambda) \geq \bfg_2(2,\lambda)$, and then define
    \begin{equation*}
        \bfg_2(A,\lambda) := \frac{A}{2} \cdot \bfg_2(2,\lambda), \text{ for all }
    A\in(0,2), \: \lambda>1,
    \end{equation*}
then we will have certainly constructed a minimizer of the collection $\scg_2$.
Inductively, we define the candidate along the boundary $A=2$ by
$$
    \bfg_2(2, \lambda) \coloneqq 
    \begin{cases}
        1, & \text{ if } \lambda \le 2, \\
        \frac{1}{2^{\lceil \lambda \rceil -2}}, & \text{ if } \lambda \ge 2,
    \end{cases}
    $$
which leads to the full expression of our candidate for the $C=2$ case (pictured in Figure \ref{fig:C2-built}):
    $$
    \bfg_2(A, \lambda) \coloneqq 
    \begin{cases}
        1, &\text{ if } \lambda \leq 0,\\
        \min(1, A), &\text{ if } 0 < \lambda \leq 1,\\
        \frac{A}{2^{\lceil \lambda \rceil -1}}, &\text{ if } \lambda > 1.
    \end{cases}
    $$

We run through one more example, where this time $C$ is not a natural number. Afterwards, the path to general $C>1$ should be clear.


\subsection{\texorpdfstring{The case $C=3.2$}{TEXT}}
Let $G\in\scg_{3.2}$ be any supersolution. Let us begin, as before, by jumping from $A=0$ (see Figure \ref{fig:concavity for C=3.2}):
$$
G(0,\lambda)=1, \: \lambda \leq 0 
\: \Rightarrow \:
G(1,\lambda) =1, \: \lambda \leq 1
\: \Rightarrow \:
G(2,\lambda)=1, \: \lambda \leq 2
\: \Rightarrow \:
G(3,\lambda)=1, \: \lambda \leq 3.
$$
We cannot jump anymore from $A=3$, but we also have not reached the boundary $A=3.2$! So, let's also jump from $A=0.2$:
 $$
G(0.2,\lambda)=1, \: \lambda \leq 0 
\: \Rightarrow \:
G(1.2,\lambda) =1, \: \lambda \leq 1
\: \Rightarrow \:
G(2.2,\lambda)=1, \: \lambda \leq 2
\: \Rightarrow \:
G(3.2,\lambda)=1, \: \lambda \leq 3.
$$

Having now reached the boundary, we use horizontal concavity to ``fill in'' the remaining parts of the domain for $\lambda\leq 3$.
    \begin{itemize}
        \item When $\lambda \in (0, 1]$, we have $G(1,\lambda)=G(3.2, \lambda)=1$, so $G(A,\lambda)=1$ for all $A \in [1,3.2]$. If $A\in (0,1)$ we obtain as before $G(A,\lambda) \geq A$.
        \item When $\lambda \in (1,2]$, we have maximality at $A=2$ and $A=3.2$, so $G(A,\lambda)=1$ for all $A\in [2,3.2]$. For $A\in (0,2)$, we obtain $G(A,\lambda) \geq \sfrac{A}{2}$.
        \item Finally, for $\lambda \in (2,3]$, we have $G(A,\lambda)=1$ for $A\in [3,3.2]$. If $A\in(0,3)$, we write
        $A = (1-\theta) \cdot 0 + \theta \cdot 3$, so $\theta=\sfrac{A}{3}$, and obtain $G(A,\lambda) \geq \sfrac{A}{3}$.
    \end{itemize}

\begin{figure}
    \centering
    \includegraphics[width=\linewidth]{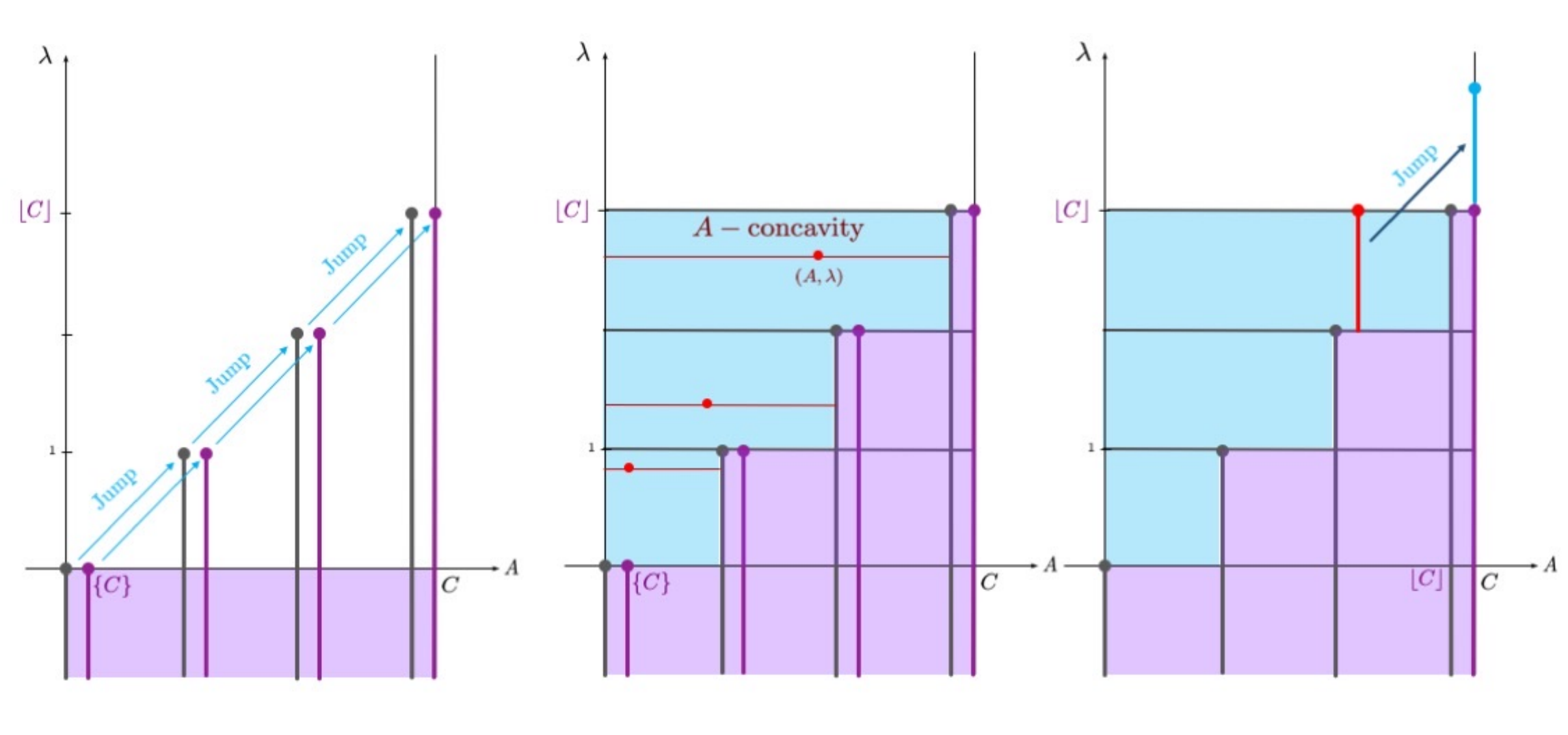}
    \caption{Constructing the candidate for non-integer $C$.}
    \label{fig:concavity for C=3.2}
\end{figure}

We now move to the region $\lambda > 3$. As the reader can anticipate by now, from this point on we can reduce the problem to constructing on the boundary $A=3.2$: if we write $A\in (0,3.2)$ as $A = (1-\theta) \cdot 0 + \theta \cdot 3.2$, we have $\theta=\sfrac{A}{3.2}$ and then
    $$
    G(A,\lambda) \geq \frac{A}{3.2} \cdot G(3.2, \lambda), \text{ for all } \lambda>3.
    $$
Furthermore, from the use of the jump inequality between $G(2.2, \lambda)$ and $G(3.2, \lambda+1)$ we see that we can recursively define our candidate along the boundary:
\begin{equation*}
    \bfg_{3.2} (3.2, \lambda) \coloneqq
    \begin{cases}
        1, &\text{ if } \lambda \le 3\\
        \frac{2.2}{3}, &\text{ if } 3 < \lambda \le 4\\
        \frac{2.2}{3}\left( \frac{2.2}{3.2} \right), &\text{ if } 4 < \lambda \le 5\\
        &\vdots\\
        \frac{2.2}{3}\left( \frac{2.2}{3.2} \right)^{n-1}, &\text{ if } 3 + (n-1) < \lambda \le 3 + n,
    \end{cases}
\end{equation*}
which gives our complete candidate function for the case $C = 3.2$ (pictured in Figure \ref{fig:C3pt2}):
\begin{equation*}
    \bfg_{3.2}(A, \lambda) \coloneqq
    \begin{cases}
        1, &\text{ if } \lambda \le 0\\
        \min\left( 1, A \right), &\text{ if } 0 < \lambda \le 1\\
        \min\left( 1, \frac{A}{2} \right), &\text{ if } 1 < \lambda \le 2\\
        \min\left( 1, \frac{A}{3} \right), &\text{ if } 2 < \lambda \le 3\\
        \frac{A}{3.2} \frac{2.2}{3}, &\text{ if } 3 < \lambda \le 4\\
        \frac{A}{3.2} \frac{2.2}{3} \left( \frac{2.2}{3.2} \right), &\text{ if } 4 < \lambda \le 5\\
        &\vdots\\
        \frac{A}{3.2} \frac{2.2}{3} \left( \frac{2.2}{3.2} \right)^{n-1}, &\text{ if } 3 + (n-1) < \lambda \le 3 + n.
    \end{cases}
\end{equation*}
For the sake of cleanliness, we rewrite this as 
\begin{equation*}
    \bfg_{3.2} (A, \lambda) \coloneqq
    \begin{cases}
        1, &\text{ if } \lambda \le 0\\
         & \\
        \min\left( 1, \frac{A}{\lceil \lambda\rceil} \right), &\text{ if } 0 < \lambda \le 3\\
         & \\
        \frac{A}{3} \left( \frac{2.2}{3.2} \right)^{\ceil{\lambda} - 3}, &\text{ otherwise}.
    \end{cases}
\end{equation*}

\begin{figure}[h]
  \centering
  \begin{subfigure}{0.45\textwidth}
    \centering
    \includegraphics[width=\linewidth]{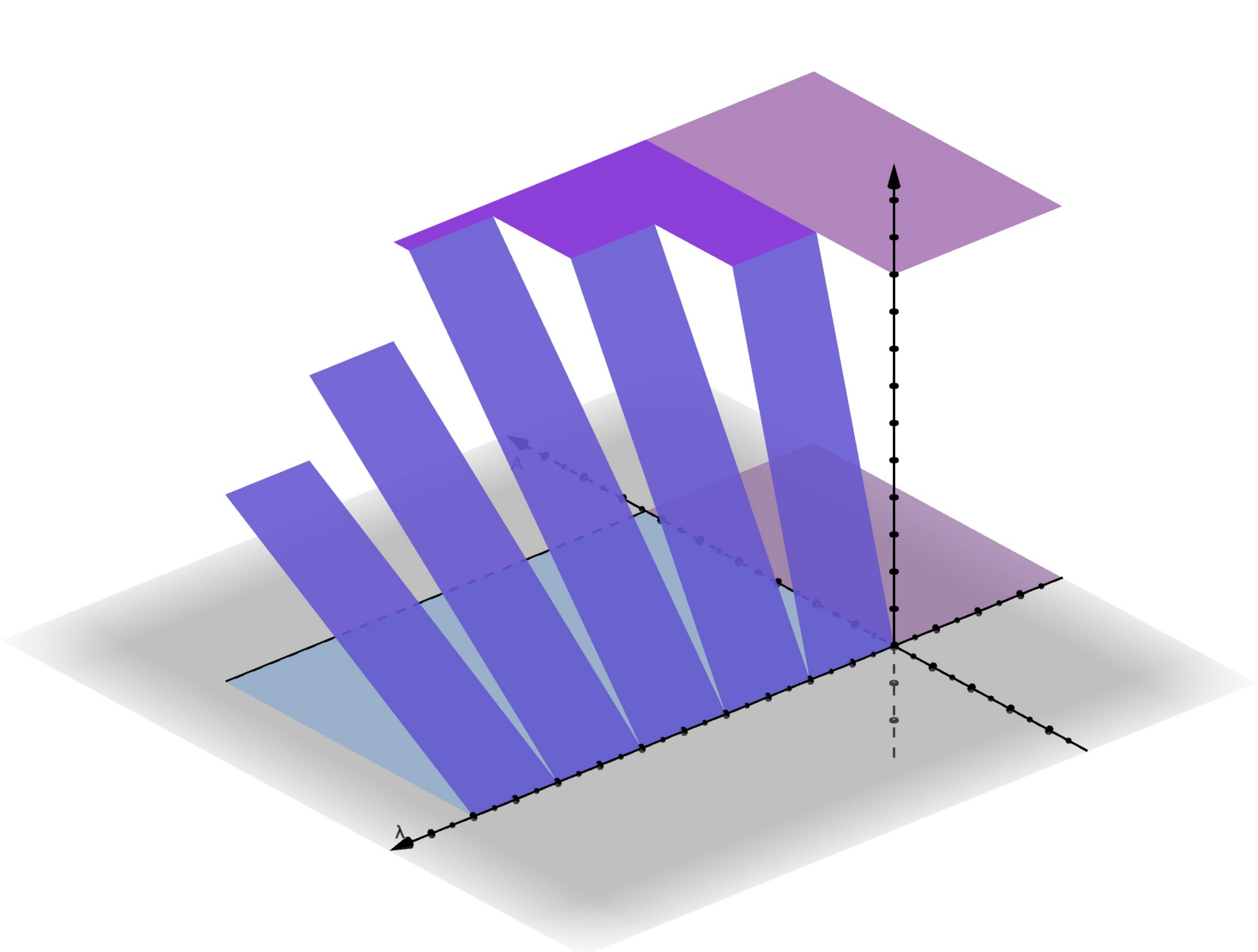}
    \caption{$C=3.2$}
    \label{fig:C3pt2}
  \end{subfigure}\hfill
  \begin{subfigure}{0.55\textwidth}
    \centering
    \includegraphics[width=\linewidth]{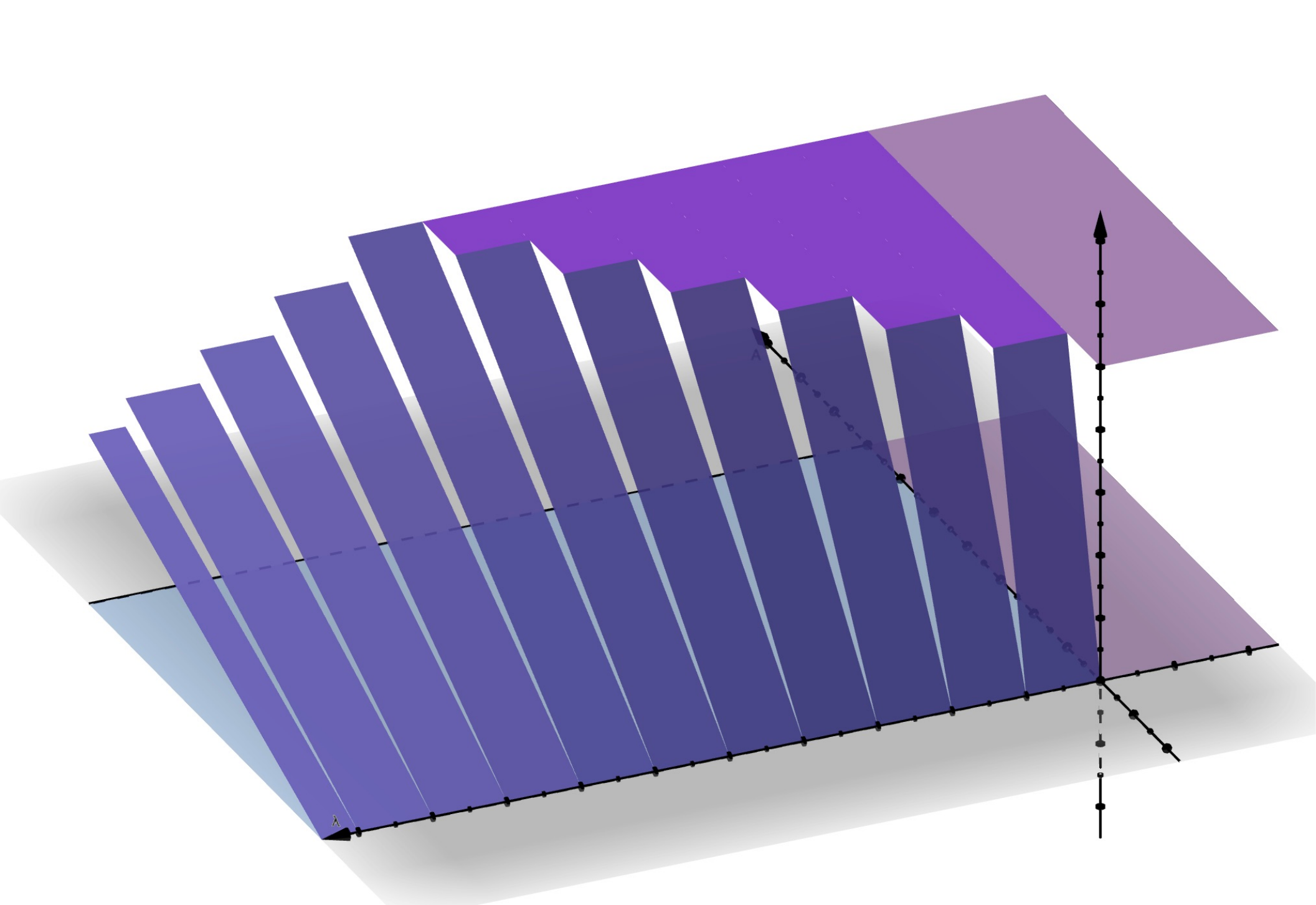}
    \caption{$C=7$}
    \label{fig:C=7 candidate}
  \end{subfigure}
  \caption{The candidate functions $\bfg_{3.2}$ and $\bfg_7$.}
  \label{fig:twoside}
\end{figure}

\begin{remark}
    Observe that for the construction, we only relied on the Obstacle Condition and the Main Inequality (the two properties which made $\mbg_C$ the least supersolution), and did not pre-assume any of the other properties we extracted for $\mbg_C$ in Theorem \ref{T:BellmanProperties}. However, they unfolded naturally out of the construction: for example, $\bfg_C$ is non-increasing in $\lambda$ and satisfies $\bfg_C(A,\lambda)=\bfg_C(A,\lceil\lambda\rceil)$.
\end{remark}

\subsection{\texorpdfstring{Generalizing for $C>1$}{TEXT}}
Suppose first that $C\geq 2$ is a positive integer. In this case, we jump $C$ times from $A=0$ and reach the boundary:
$$
G(0,\lambda)=1, \: \lambda \leq 0 
\Rightarrow
G(1,\lambda)=1, \: \lambda \leq 1
\Rightarrow
\ldots
\Rightarrow
G(C,\lambda)=1, \: \lambda\leq C.
$$
Filling in via $A$-concavity, we obtain
\begin{equation*}
    G(A,\lambda) \geq \bfg_C(A, \lambda) \coloneqq
    \begin{cases}
        1, \quad &\text{ if } \lambda \le 0\\
        \min\left( 1, \frac{A}{\ceil{\lambda}} \right), \quad &\text{ if } 0<\lambda\leq C.
    \end{cases}
\end{equation*}

For $\lambda>C$, we can interpolate directly between the $A=0$ and $A=C$ boundaries and obtain
    \begin{equation}
        \label{e:reducetoboundary}
    G(A,\lambda) \geq \frac{A}{C} G(C,\lambda),
    \end{equation}
leaving us to construct the candidate for $A=C$. We apply jump-concavity recursively, jumping from $C-1$:
    \begin{align*}
    G(C-1,\lambda) \geq \frac{C-1}{C}, \: \lambda\in(C-1,C] 
    & \Rightarrow
    G(C,\lambda) \geq \frac{C-1}{C}, \: \lambda\in(C,C+1]\\
    & \Rightarrow 
    G(C-1,\lambda)\geq \left(\frac{C-1}{C}\right)^2, \: \lambda\in(C,C+1]\\
    & \Rightarrow
    G(C,\lambda) \geq \left(\frac{C-1}{C}\right)^2, \: \lambda\in(C+1,C+2]\\
    & \vdots
    \end{align*}
Inductively, 
\begin{equation*}
    G(A,\lambda) \geq \bfg_C(A, \lambda) \coloneq
    \begin{cases}
        1, \quad &\text{ if } \lambda \le 0\\
        \min\left( 1, \frac{A}{\ceil{\lambda}} \right), \quad &\text{ if } 0 < \lambda \le C\\
        \frac{A}{C} \left( \frac{C-1}{C} \right)^{\ceil{\lambda} - C}, \quad &\text{ otherwise}.
    \end{cases}
\end{equation*}

Finally, suppose $C$ is not necessarily an integer. 
However, note that in both cases, we first define our candidate for $\lambda \le \floor{C}$, and then for $\lambda > \floor{C}$.
The only adjustment we need to make for the case $\lambda \le \floor{C}$ is that we run the series of $C$-many ``jumps from the Obstacle'' twice, once from $A=0$ and again from $A=\{C\}$. For $\lambda > \floor{C}$, we interpolate between the $A=0$ and $A=C$ boundaries, as in \eqref{e:reducetoboundary}, and construct the candidate along the $A=C$ boundary by recursively jumping from $A=C-1$.  
Thus, we reconcile the cases $C=\floor{C}$ and $C \neq \floor{C}$ to arrive at the candidate Bellman function as follows.
\begin{equation}\label{eq:COMPLETE CANDIDATE}
    \bfg_C(A, \lambda) \coloneq
    \begin{cases}
        1, \quad &\text{ if } \lambda \le 0\\
        &\\
        \min\left( 1, \frac{A}{\ceil{\lambda}} \right), \quad &\text{ if } 0 < \lambda \le \floor{C}\\
        &\\
        \frac{A}{\floor{C}} \left( \frac{C-1}{C} \right)^{\ceil{\lambda} - \floor{C}}, \quad &\text{ otherwise}.
    \end{cases}
\end{equation}

\section{\texorpdfstring{Proving the candidate $\bfg_C$ satisfies the Main Inequality}{TEXT}} 
\label{sec:true candidate}
Now that we have a candidate Bellman function, $\bfg_C$, constructed such that $\bfg_C \le G$ for all $G \in \scg_C$, we want to show that $\bfg_C$ belongs to the collection $\scg_C$. This will prove $\bfg_C = \mbg_C$. We assume $C\geq 1$ is fixed from now on, and simply write $\bfg=\bfg_C$ for the remainder of this section.

Since $\bfg$ satisfies the Obstacle Condition by construction, we are only left to show that it satisfies the Main Inequality, that is, the candidate satisfies both midpoint concavity and the jump inequality.

\begin{lemma}\label{lem:proof of midpoint concavity}
    Let $\bfg(A, \lambda)$ be defined as in \eqref{eq:COMPLETE CANDIDATE}. Then
    $$
        \bfg\left( \frac{A_1 + A_2}{2}, \lambda \right) \ge \frac{1}{2}\left( \bfg(A_1, \lambda) + \bfg(A_2, \lambda) \right),
    $$
    for $0 \le A_1, A_2 \le C$ and all $\lambda \in \mbr$.
\end{lemma}
\begin{proof}
    Let $0 \le A_1, A_2 \le C$. We prove by cases.

\begin{enumerate}
    \item Suppose $\lambda \le 0$. Then
    $$
        \bfg\left( \frac{A_1 + A_2}{2}, \lambda \right) = 1 = \frac{1}{2} (1 + 1) = \frac{1}{2}\left( \bfg(A_1, \lambda) + \bfg(A_2, \lambda) \right).
    $$

    \item Suppose $0 < \lambda \le \floor{C}$.
    \begin{enumerate}
        \item If $\min\left( 1, \frac{A_1}{\ceil{\lambda}} \right) = \min\left(  1, \frac{A_2}{\ceil{\lambda}} \right) = 1$, then $\frac{1}{2}\left( \bfg(A_1, \lambda) + \bfg(A_2, \lambda) \right) = 1$.
    
        Furthermore, $\bfg\left( \frac{A_1 + A_2}{2}, \lambda \right) = \min\left( 1, \frac{1}{2}\left( \frac{A_1 + A_2}{\ceil{\lambda}} \right) \right)$. 
        
        As we have both $\frac{A_1}{\ceil{\lambda}} \ge 1$ and $\frac{A_2}{\ceil{\lambda}} \ge 1$, we arrive at $\frac{A_1 + A_2}{\ceil{\lambda}} \ge 2$ and so $\bfg\left( \frac{A_1 + A_2}{2}, \lambda \right) = 1$.

        \item Without loss of generality, consider if $\min\left( 1, \frac{A_1}{\ceil{\lambda}} \right) = 1$ and $\min\left(  1, \frac{A_2}{\ceil{\lambda}} \right) = \frac{A_2}{\ceil{\lambda}}$.

        Then we have both
        $$
            \frac{1}{2}\left( \bfg(A_1, \lambda) + \bfg(A_2, \lambda) \right) = \frac{1}{2} \left( 1+\frac{A_2}{\ceil{\lambda}} \right) \le 1
        $$
        as $\frac{A_2}{\ceil{\lambda}} \le 1$, and
        $$
            \frac{1}{2}\left( \bfg(A_1, \lambda) + \bfg(A_2, \lambda) \right) = \frac{1}{2} \left( 1+\frac{A_2}{\ceil{\lambda}} \right) \le \frac{1}{2}\left( \frac{A_1 + A_2}{\ceil{\lambda}}\right)
        $$
        as $1 \le \frac{A_1}{\ceil{\lambda}}$.
        
        Thus, $\frac{1}{2}\left( \bfg(A_1, \lambda) + \bfg(A_2, \lambda) \right) \le \min\left( 1, \frac{1}{2}\left( \frac{A_1 + A_2}{\ceil{\lambda}} \right) \right) = \bfg\left( \frac{A_1 + A_2}{2}, \lambda \right)$.

        \item If $\min\left( 1, \frac{A_1}{\ceil{\lambda}} \right) = \frac{A_1}{\ceil{\lambda}}$ and $\min\left(  1, \frac{A_2}{\ceil{\lambda}} \right) = \frac{A_2}{\ceil{\lambda}}$ then
        \begin{align*}
            \frac{1}{2}\left( \bfg(A_1, \lambda) + \bfg(A_2, \lambda) \right) &= \frac{1}{2}\left( \frac{A_1}{\ceil{\lambda}} + \frac{A_2}{\ceil{\lambda}} \right)\\
            &= \min\left( 1, \frac{1}{2}\left( \frac{A_1 + A_2}{\ceil{\lambda}} \right) \right) = \bfg\left( \frac{A_1 + A_2}{2}, \lambda \right).
        \end{align*}
    \end{enumerate}

    \item Finally, suppose $\lambda \ge \floor{C}$. Then
    \begin{align*}
        \frac{1}{2}\left( \bfg(A_1, \lambda) + \bfg(A_2, \lambda) \right) &= \frac{1}{2}\left( \frac{A_1}{\floor{C}}\left( \frac{C-1}{C} \right)^{\ceil{\lambda}-\floor{C}} + \frac{A_2}{\floor{C}}\left( \frac{C-1}{C} \right)^{\ceil{\lambda}-\floor{C}} \right)\\
        &= \frac{A_1 + A_2}{2\floor{C}} \left( \frac{C-1}{C} \right)^{\ceil{\lambda}-\floor{C}}\\
        &= \bfg\left( \frac{A_1 + A_2}{2}, \lambda \right).
    \end{align*}
\end{enumerate}
    Thus, we see that the candidate satisfies midpoint concavity in all cases.
\end{proof}

\begin{lemma}\label{lem:proof of jump inequality}
    Let $\bfg(A, \lambda)$ be defined as in \eqref{eq:COMPLETE CANDIDATE}. Then
    $$
        \bfg\left( A+1, \lambda+1\right) \ge \bfg(A, \lambda),
    $$
    for $0 \le A \le C-1$ and all $\lambda \in \mbr$.
\end{lemma}
\begin{proof}
    Let $0 \le A \le C-1$. We consider the following cases.

    \begin{enumerate}
        \item Let $\lambda \le 0$ and $\lambda + 1 \le 0$. Then $\bfg(A+1, \lambda+1) = 1 = \bfg(A, \lambda)$.
        
        \item Let $\lambda \le 0$ and $0 < \lambda + 1 \le \floor{C}$. 
        \begin{enumerate}
            \item If $\min\left( 1, \frac{A+1}{\ceil{\lambda+1}} \right) = 1$, then $\bfg(A+1, \lambda+1) = 1 = \bfg(A, \lambda)$. 
            
            \item If $\bfg(A+1, \lambda+1) = \min\left( 1, \frac{A+1}{\ceil{\lambda+1}} \right) = \frac{A+1}{\ceil{\lambda+1}}$ and $\bfg(A, \lambda) = 1$ then
            $$
                \frac{A+1}{\ceil{\lambda+1}} - 1 = \frac{A+1 - \ceil{\lambda+1}}{\ceil{\lambda+1}} = \frac{A - \ceil{\lambda}}{\ceil{\lambda+1}} = A > 0,
            $$ as $\lambda \le 0$ and $0<\lambda+1$ together imply $\ceil{\lambda} = 0$.
        \end{enumerate}
        
        \item Let $0 < \lambda \le \floor{C}$ and $0 < \lambda + 1 \le \floor{C}$. Then $\bfg(A+1, \lambda+1) = \min\left( 1, \frac{A+1}{\ceil{\lambda + 1}} \right)$ and $\bfg(A,\lambda) = \min\left( 1, \frac{A}{\ceil{\lambda}} \right)$.
        \begin{enumerate}
            \item If $\bfg(A+1, \lambda+1) = 1 = \bfg(A, \lambda)$, then the jump inequality holds.
            
            \item If $\bfg(A+1, \lambda+1) = 1 \ge \frac{A}{\ceil{\lambda}} = \bfg(A, \lambda)$, then the inequality holds.
            
            \item If $\bfg(A+1, \lambda+1) = \frac{A+1}{\ceil{\lambda + 1}}$ and $\bfg(A, \lambda) = 1$, then these, respectively, allow us to see that $A+1 \le \ceil{\lambda + 1}$ and $A \ge \ceil{\lambda}$. Together these imply $A = \ceil{\lambda}$. Thus, we can conclude $\frac{A+1}{\ceil{\lambda+1}} = 1$ and the jump inequality holds.
            
            \item If $\bfg( A+1, \lambda+1 ) = \frac{A+1}{\ceil{\lambda + 1}}$ and $\bfg( A, \lambda ) = \frac{A}{\ceil{\lambda}}$ then
            \begin{align*}
                \frac{A+1}{\ceil{\lambda + 1}} - \frac{A}{\ceil{\lambda}}&= \frac{(A+1)\ceil{\lambda} - A\ceil{\lambda+1}}{(\ceil{\lambda+1})(\ceil{\lambda})}\\
                &= \frac{\ceil{\lambda} - A}{(\ceil{\lambda+1})(\ceil{\lambda})} \ge 0,
            \end{align*}
            since  $A \le \ceil{\lambda}$.
        \end{enumerate}
        
        \item Let $0 < \lambda \le \floor{C}$ and $\lambda + 1 > \floor{C}$. 
        Then $\bfg(A+1, \lambda+1) = \frac{A+1}{\floor{C}}\left( \frac{C-1}{C} \right)^{\ceil{\lambda+1} - \floor{C}}$ and $\bfg(A, \lambda) = \min\left( 1, \frac{A}{\ceil{\lambda}} \right) = \frac{A}{C}$.
        We note this conclusion about $\bfg(A, \lambda)$ arises because we have $\lambda \le \floor{C}$ and $\floor{C} < \lambda + 1$, which together imply $\ceil{\lambda} = \floor{C}$. Additionally, $A \le C-1$ so we have the minimum as described.
        This being the case, we show $\frac{A+1}{\floor{C}}\left( \frac{C-1}{C} \right)^{\ceil{\lambda+1} - \floor{C}} \ge \frac{A}{C}$. 
        
        We know $\ceil{\lambda+1} = \floor{C}+1$, so $\bfg(A+1, \lambda+1) = \frac{A+1}{\floor{C}}\left( \frac{C-1}{C} \right)$. Now,
        \begin{align*}
            \frac{A+1}{\floor{C}}\left( \frac{C-1}{C} \right) \frac{C}{A} &= \frac{A+1}{A} \left( \frac{C-1}{C} \right)\frac{C}{\floor{C}}\\
            &= \left( 1 + \frac{C-(A+1)}{AC} \right) \frac{C}{\floor{C}} \ge 1,
        \end{align*}
        as $A+1 \le C$.
        Thus, $\bfg(A+1, \lambda+1) \ge \frac{A}{C} = \bfg(A, \lambda)$.
        
        \item Let $\lambda > \floor{C}$ and $\lambda + 1 > \floor{C}$. Then 
        \begin{align*}
            \frac{\bfg(A+1, \lambda+1)}{\bfg(A,\lambda)} &= \frac{A+1}{\floor{C}} \frac{\floor{C}}{A}\left( \frac{C-1}{C} \right)^{\ceil{\lambda+1} - \floor{C}}\left( \frac{C-1}{C} \right)^{\floor{C} - \ceil{\lambda}}\\
            &= \frac{A+1}{A}\left( \frac{C-1}{C} \right)\\
            &= 1 + \frac{C - (A+1)}{AC} \ge 1.
        \end{align*}
    \end{enumerate}
Therefore, we see that the candidate $\bfg$ as defined in \eqref{eq:COMPLETE CANDIDATE} satisfies the jump inequality.
\end{proof}



\bibliographystyle{abbrv}

\section*{Acknowledgments}
Part of this work was completed during the Spring 2024 semester at Texas A\&M University, as part of the Undergraduate Research Topics course MATH 491. The authors would like to thank Guillermo Rey, Kristina Ana Škreb, and Patricia Alonso Ruiz for many helpful conversations and inspiration. 
I.H.F. is supported by NSF grant NSF-DMS-2246985. 

\vspace{0.2in}

\small
\begin{tabular}{lll}
\textbf{Shivam Aggarwal} & Texas A\&M University & \texttt{shivamaggarwal@tamu.edu} \\
\textbf{Samuel Hernandez} & Texas A\&M University & \texttt{samuelhq11@tamu.edu} \\
\textbf{Irina Holmes Fay} & University of Wyoming & \texttt{iholmesf@uwyo.edu} \\
\textbf{Jennifer Mackenzie} & Texas A\&M University & \texttt{jennifer.mackenzie2@tamu.edu}
\end{tabular}

\end{document}